\documentclass[a4paper,12pt,reqno]{amsart}
\usepackage{txfonts}
\usepackage{amssymb}
\usepackage[mathscr]{eucal}
\usepackage{amsmath, amssymb}
\usepackage{amscd,amsthm}
\usepackage[dvips]{color}
\usepackage{multicol}
\usepackage[all]{xy}
\usepackage{graphicx}
\usepackage{soul}

\usepackage{color}
\usepackage{colordvi}
\usepackage{xspace}
\usepackage[colorinlistoftodos]{todonotes}

\usepackage[colorlinks,final,hyperindex]{hyperref}

\usepackage{tikz}
\usepackage[active]{srcltx}
\usetikzlibrary{matrix,arrows}

\newtheorem{thm}{Theorem}[section]
\newtheorem{cor}[thm]{Corollary}
\newtheorem{lem}[thm]{Lemma}
\newtheorem{prop}[thm]{Proposition}

\theoremstyle{definition}
\newtheorem{defn}[thm]{Definition}

\newtheorem{rk}[thm]{Remark}
\newtheorem{ex}[thm]{Example}
\newtheorem{defnlem}[thm]{Definition-Lemma}
\newtheorem{defnprop}[thm]{Definition-Proposition}
\DeclareMathOperator{\End}{End}        %% Endomorphism
\DeclareMathOperator{\Op}{Op}      %%
\DeclareMathOperator{\spec}{spec}           %% spectrum
\DeclareMathOperator{\supp}{supp}      %% support
\DeclareMathOperator{\diam}{diam}      %% diametro

\DeclareMathOperator{\TR}{TR}                 %% canonical trace of operator
\newcommand{\s}{\sigma}                       %% short for \beta
\newcommand{\G}{\Gamma}                       %% short for \Gamma
\newcommand{\ve}{\varepsilon}                       %% short for \varepsilon

\newcommand{\C}{\mathbb{C}}              %% complex numbers
\newcommand{\D}{\mathbb{D}}               %% Unit disc
\newcommand{\N}{\mathbb{N}}
\newcommand{\Z}{\mathbb{Z}}

\newcommand{\R}{\mathbb{R}}

\newcommand{\T}{\mathbb{T}}
\renewcommand {\s} {\sigma}

\newcommand{\cutoffint}{-\hskip -10pt\int}

\newcommand {\e} {{\varepsilon}}

\newcommand {\Ci} {{C^\infty}}

\newcommand{\wt}{\widetilde}

\newcommand{\la}{\lambda}                %% Short for lambda

\newcommand{\Upsi}{{\mathcal U}\Psi}
\newcommand{\ds}{\displaystyle}

\newcommand\reallywidehat[1]{\arraycolsep=0pt\relax%
\begin{array}{c}
\stretchto{
  \scaleto{
    \scalerel*[\widthof{\ensuremath{#1}}]{\kern-.5pt\bigwedge\kern-.5pt}
    {\rule[-\textheight/2]{1ex}{\textheight}} %WIDTH-LIMITED BIG WEDGE
  }{\textheight} % 
}{0.5ex}\\           % THIS SQUEEZES THE WEDGE TO 0.5ex HEIGHT
#1\\                 % THIS STACKS THE WEDGE ATOP THE ARGUMENT
\rule{-1ex}{0ex}
\end{array}
}

\usepackage{scalerel}%[2014/03/10]
\usepackage{stackengine}

\newcommand\reallywidetilde[1]{\ThisStyle{%
  \setbox0=\hbox{$\SavedStyle#1$}%
  \stackengine{-.1\LMpt}{$\SavedStyle#1$}{%
    \stretchto{\scaleto{\SavedStyle\mkern.2mu\sim}{.5467\wd0}}{.7\ht0}%
  }{O}{c}{F}{T}{S}%
}}

\begin{document}

\title{Spectral $\zeta$-invariants lifted to  coverings}

\author{Sara  Azzali}
\address{Institute of Mathematics, Universit\"at Potsdam, 
 Campus II - Golm, Haus 9\\
Karl-Liebknecht-Stra{\ss}e 24-25\\
D-14476 Potsdam, Germany }
\email{azzali@uni-potsdam.de}

\author{Sylvie Paycha}
\address{Institute of Mathematics, Universit\"at Potsdam, 
 Campus II - Golm, Haus 9\\
Karl-Liebknecht-Stra{\ss}e 24-25\\
D-14476 Potsdam, Germany\\
(On leave from the Universit\'e Blaise Pascal, Clermont-Ferrand)}
\email{paycha@math.uni-potsdam.de}
\subjclass{Primary 47G30, 58J42, 	58J40 ; Secondary  58J28,19K56}

\keywords{Locality, Zeta-regularised traces, Pseudodifferential Operators, Wodzicki Residue, $L^2$-invariants}

\begin{abstract}

  \small{The canonical trace and the Wodzicki residue on classical pseudodifferential operators on a closed manifold are characterised by their locality and shown to be preserved under lifting to  the universal covering as a result of their local feature.  
 As a consequence, we lift a class of spectral $\zeta$-invariants using lifted defect formulae which express discrepancies of  $\zeta$-regularised traces 
 in terms of Wodzicki residues.  We derive Atiyah's $L^2$-index theorem as an instance of the $\Z_2$-graded generalisation of the canonical lift of spectral $\zeta$-invariants and we show that certain  lifted spectral $\zeta$-invariants for geometric operators are integrals of Pontryagin and Chern forms.} 

\end{abstract}

\maketitle 

\section*{Introduction}
%\addcontentsline{toc}{section}{Introduction}

A differential operator $A$ on a closed manifold $M$ 
lifts to a differential operator $\wt A$ on its universal covering $\wt M\to M$. This is due to the locality property of  differential operators which preserve the support of the sections they act  on. This lifting property does not extend to  general pseudodifferential operators  which are  only pseudo-local i.e., they only preserves the singular support of the sections. Hence arises the problem of lifting complex powers $ Q^{-z}$ involved in spectral $\zeta$-functions
\begin{equation}\zeta_{A,Q}(z):={\rm TR}(A\, Q^{-z})\ .\label{eq:zeta1}\end{equation}
where $A$ and $Q$ are differential operator on $M$, and ${\rm TR}$ is the canonical trace.
Nevertheless, we prove  that spectral $\zeta$-invariants 	\begin{equation}\label{eq:zetaintro}\zeta_{A,Q}(0):={\rm fp}_{z=0}\left( {\rm TR}(A\, Q^{-z})\right)\end{equation} corresponding to the constant term of the Laurent expansion of \eqref{eq:zeta1}   canonically lift to coverings (see (\ref{eq:diffliftedregtraces}) in Theorem \ref{thm:liftedregtraces}).
   This results from a detailed  analysis of the intertwining between the pseudo-locality of the complex powers and the locality of the canonical trace on non-integer order operators.  In our approach,  the locality  of spectral $\zeta$-invariants is only an instance of the more general locality  expressed by defect formulae. Another central result of the paper  are $L^2$-counterparts of such defect formulae (\ref{eq:introlifteddefect2}). A natural application is the locality of  Atiyah's $L^2$-index, which is expressed as a $\Gamma$-Wodzicki residue (\ref{eq:indres2}). \\
   Our approach can be summarised as follows. We build: 
   	\begin{itemize}
		\item  a holomorphic germ of pseudifferential (and hence pseudo-local) operators $A(z)$ (of holomorphic order $\alpha(z)$) which at zero is the differential (and hence {\it local}) operator $A$;
		\item the corresponding meromorphic germ ${\rm TR}(A(z))$ of functions built from the {\it local linear form} given by the canonical trace;
		\item the {\it local} invariant is obtained as the value at $z=0$ of this germ of functions in terms of the $1$-jet of the germ of operators 
		\[\lim_{z\to 0}{\rm TR}(A(z))=-\frac{1}{\alpha^\prime(0)}\,{\rm Res}(A^\prime(0)).\]
		\end{itemize} 
   
To achieve our goal,  following Shubin \cite{Shub}, we view pseudodifferential operators as \lq\lq small perturbations\rq\rq{}   of  pseudodifferential operator with finite propagation, more precisely those which are  $\ve$-local for some small enough $\ve>0$. Such operators, which fall into the more general class of quasi-local operators introduced by Roe \cite{Roe2}, are properly supported and hence determined by their symbol.  Quasi-local operators  can   roughly  be viewed as operators with controlled propagation at infinity, see \cite{E} for a detailed discussion. They bare the advantage over finite propagation operators, that they are stable under functional calculus, a property which is not needed here.  An $\ve$-local pseudodifferential operator modifies the support of the sections it is acting on, by at most the distance $\ve$. 
The fact that a differential operator  preserves the supports is therefore confirmed by the fact that it is  $0$-local.\\
Choosing $\ve$ small enough, one can lift without ambiguity an $\ve$-local operator $A_0$ to an $\ve$-local operator $\wt A_0$. The lifted operator  is a (uniformly) properly supported pseudodifferential operator, and hence also defined in terms of its symbol $\sigma(\wt A_0)$, which is the lifted     symbol  $\sigma(\wt A_0)=  \reallywidetilde{\sigma( A_0)}$ of the original operator.\\  To go from $\ve$-local to a general classical pseudodifferential operator $A$ on a closed manifold, we observe that the latter differs from a   $\ve$-local classical pseudodifferential operator $A_0$ by an operator   with  smooth kernel supported outside the diagonal (see Proposition \ref{prop:ShubinTh1}),  so it lies in the  equivalence class  $[A_0]_{\rm diag}$ of $A_0$ for the equivalence relation  on classical pseudodifferential operators on the base manifold $M$
$$A\underset{ \rm \small diag}{\sim} B\Longrightarrow A-B \quad \text{has a  smooth kernel supported outside the diagonal}.$$  
Clearly, $\sigma(A)\sim \sigma(A_0)$ and $\reallywidetilde{\sigma(A)}\sim  \sigma(\wt A_0)$ for any $ A $ in $[A_0]_{\rm diag}.$    
 The following 
diagramme, where  ${\mathcal A}$, ${\mathcal A}_0$ are $\Gamma$-invariant operators on the covering, $\Gamma$ being the fundamental group, $\pi_\sharp {\mathcal A }_0$  the   projected   operator onto the base manifold,    presents the notations in a compact form:

\begin{equation}\label{eq:diagramme}
\begin{xy}
  \xymatrix{
     \mathcal A  \underset{ \rm \tiny diag}{\sim} \hspace{-1cm}  & {\mathcal A}_0:= \wt{A_0} \ar[d]^{\pi_\sharp} & \hspace{-.7cm} (\ve-\text{local})\\
     A \underset{ \rm \tiny diag}{\sim} \hspace{-1cm} &   A_0:=\pi_\sharp {\mathcal A}_0 \ar@<1ex>[u]^{\pi^\sharp}& \hspace{-.7cm}(\ve-\text{local})
  }
\end{xy}
\end{equation}

Note that the operator $\pi_\sharp\colon {\mathcal A}_0\longrightarrow \left(s\longmapsto \pi_*\left({\mathcal A}_0\pi^*(s)\right)\right)$ is well-defined in view of the equivariance and $\ve$-locality of ${\mathcal A}_0$ (see Proposition \ref{prop:ShubinTh1})).

Alongside the pseudo-locality of pseudodifferential operators, the other essential ingredient in our approach is the use of  {\bf local}  linear forms defined on a class of classical pseudodifferential operators   (see Definition \ref{defn:locallambda}). These only detect the symbol of an operator  and  are therefore constant on  equivalence classes $[A_0]_{\rm diag}$ and    hence constant along the horizontal lines of the above diagramme. 
{   More precisely,  a  local linear form $\Lambda$  reads:}
\begin{equation}\label{eq:Lambdaintro}
\Lambda( [A]_{\rm diag}):=\Lambda(A)=\int_M \Lambda_x(A)\, dx:=\int_M \lambda\left({\rm tr}\left(\sigma(A)(x,\cdot)\right)\right)\,dx,
\end{equation}
   $\lambda$ being a linear form on an appropriate class of scalar valued symbols,   ${\rm tr}$   the fibrewise trace on the endomorphism bundle in which  the symbol $\sigma(A)(\cdot,\xi)$ of $A$ lies for any $\xi$ in the cotangent bundle to the underlying manifold.\begin{itemize}
   	\item
     Our first main result is Theorem \ref{thm:uniqueness}, which states that any  continuous  local linear form on the class  of classical pseudodifferential operators with integer order (resp. on the class of classical pseudodifferential operators   non-integer order)
   is proportional to the  Wodzicki residue Res (see \eqref{eq:Resint}) (resp. the canonical trace TR (see \eqref{eq:TRint}))
$$
{\rm Res}(A)=\int_M {\rm Res}_x(A)\, dx;\quad \left(\text{resp.}\quad {\rm TR}(A)=\int_M {\rm TR}_x(A)\, dx\right), 
$$
 acting respectively  on the algebra of classical pseudodifferential operators with integer order and on the class of classical pseudodifferential operators with  non-integer order. The densities ${\rm Res}_x(A)\, dx$ (resp. ${\rm TR}_x(A)\, dx$) are defined  by means of the residue res (resp. the canonical integral $\cutoffint_{\R^n}$) on integer order (resp. on non-integer order) scalar symbols.
 
 A local linear form  (\ref{eq:Lambdaintro})  can be lifted from a $\ve$-local operator $  A_0$ on $M$  to its lift $\wt A_0$ on $\wt M$  by \begin{equation}
 \label{eq:Lambdaintlift}
  \Lambda_\Gamma(\wt A_0)= \int_F \lambda\left({\rm tr}\left( \sigma(\wt A_0)(x,\cdot)\right)\right)\,dx,
  \end{equation} where $F$ is a fundamental domain for the action of the fundamental group. It  further lifts  to any $\Gamma$-invariant operator $\mathcal A$  on the covering; indeed $\mathcal A$ lies in the class $[\wt A_0]_{\rm diag}$ of some lifted $\ve$-local operator $\wt{A_0}$. Since $ \Lambda_\Gamma$ is constant on such a class, we set \begin{equation}\label{eq:Lambdawt}
 \Lambda_\Gamma(\mathcal A):= \int_F \lambda\left({\rm tr}\left( \sigma(\wt A_0)(x,\cdot)\right)\right)\,dx.
\end{equation}  
Prototypes are 
the $\Gamma$-residue (resp. $\Gamma$-canonical trace) (see Proposition \ref{prop:ResTRliftedA})
$${\rm Res}_\Gamma(\mathcal A):= \int_F{\rm Res}_{\wt x}(\wt A_0)\,d\wt x,\quad \left(\text{resp.} {\rm TR}_\Gamma(\mathcal A):= \int_F{\rm TR}_{\wt x}(\wt A_0)\,d\wt x\,\right),$$
obtained from integrating the residue and canonical trace densities on $F$.

Whereas the canonical trace lifts to coverings due to its local feature,  the   {\bf regularised trace} evaluated at $z=p$ of   a  holomorphic family $A(z)$ of classical pseudodifferential operators on $M$, defined as the   Hadamard finite part
$$
{\rm fp}_{z=p} {\rm TR}(A(z)):=\lim_{z\to p}\left({\rm TR}(A(z))-\frac{{\rm Res}_{z=p}\left({\rm TR}(A(z))\right)}{z-p} \right),
$$
  (here ${\rm Res}_{z=p}$ stands for the complex residue at $p$)    is  generally non local  and does not a priori lift to coverings. However,  {\bf defect formulae}, which express the  discrepancies of regularised traces  in terms of  the Wodzicki residue  and therefore also enjoy a local feature,  do  lift to   coverings (Theorem \ref{thm:KVPScov}).  More precisely, if $A(p)$ has a well-defined canonical trace TR$(A(p))$, the trace defect formula (Theorem \ref{thm:KVPS}, borrowed from \cite{KV} and \cite{PS}), relates the regularised trace $ {\rm fp}_{z=p} {\rm TR}(A(z)) $ 
 with the (extended) residue ${\rm Res}\left(\log A^\prime(p)\right)$ of the derivative of the family at this pole (see (\ref{eq:PSclassicalop})),
\begin{equation}\label{eq:introdefect}{\rm fp}_{z=p} {\rm TR}(A(z))= {\rm TR}(A(p))+\frac{1}{q}\,{\rm Res}\left( A^\prime(p)\right),\end{equation}
where the operators $A(z)$ have order $a-qz$ for some   given positive $q$.\\
If $A(p)$ is  a differential operator, then ${\rm TR}(A(p))=0$ and  (\ref{eq:introdefect}) reduces to a local expression of the regularised trace \begin{equation}
\label{eq:defectformulaintro}
{\rm fp}_{z=p} {\rm TR}(A(z))= \frac{1}{q}\,{\rm Res}\left(\log A^\prime(p)\right)\end{equation} in terms of the Wodzicki residue.\\
The trace defect formula \eqref{eq:defectformulaintro} is central to our approach since it relates regularised traces (on the l.h.s. of the above formula)  with  Wodzicki residues (on the r.h.s. of the above formula) and yields index type theorems as an application. Wodzicki residues, which only depend on one homogeneous component of the symbol and not the whole symbol, ar local. In contrast,    regularised traces built from the canonical trace, apriori depend on the whole symbol so are not expected to be local.

\item Our second main result is Theorem \ref{thm:KVPScov} which     yields the lifted analogue of the (more general) trace-defect formula (\ref{eq:introdefect})
\begin{equation}\label{eq:introlifteddefect}{\rm fp}_{z=p} {\rm TR}_\Gamma(\mathcal A(z))= {\rm TR}_\Gamma(\mathcal A(p))+\frac{1}{q}{\rm Res}_\Gamma\left(  \mathcal A^\prime(p)\right),\end{equation} for a holomorphic family   $\mathcal A(z)$   of $\Gamma$-invariant operators on the covering, such that      $\mathcal A(p)$ at the point $p$ has a   well-defined $\Gamma$-canonical trace ${\rm TR}_\Gamma(\mathcal A(p))$. If $\mathcal A(p)$ is a differential operator, then (\ref{eq:introlifteddefect}) reduces to \begin{equation}\label{eq:introlifteddefect2}{\rm fp}_{z=p} {\rm TR}_\Gamma(\mathcal A(z))= \frac{1}{q}{\rm Res}_\Gamma\left(  \mathcal A^\prime(p)\right).\end{equation}
 Corollary  \ref{cor:KVPScomparison}, which is useful for applications, then says that if $\mathcal  A(z)$ is a holomorphic family on the covering and if there exists a   holomorphic family $A(z)$  of $\ve$-local operators  on $M$, such that the difference $\mathcal  A (p)-\wt{A(p)}$  at a point $p$ has a smooth kernel, then the map $z\mapsto {\rm TR}_\Gamma\left(\mathcal  A(z)\right)- {\rm TR}_\Gamma\left(\wt{A(z)}\right)$ is holomorphic at point $p$
and
\begin{equation}\label{eq:IntroRestildeAz}
{\rm fp}_{z=p}{\rm TR}_\Gamma\left(\mathcal  A (z)\right)-{\rm fp}_{z=p}{\rm TR}\left(A(z)\right)= {\rm Tr}_\Gamma(\mathcal  A(p)-\wt{A(p)}).
%+\frac{1}{q}\,{\rm Res}_\Gamma\left((\widetilde A)^\prime(p)-\wt{A^\prime(p)}\right)
\end{equation} 
We apply (\ref{eq:IntroRestildeAz}) to the holomorphic families  $A(z)= P(\mathbf D)\, h( \mathbf  \Delta)\, Q_\ve( \mathbf \Delta)^{-z}$ on $M$ and  \hfill \break \noindent $  \mathcal A(z)=P(\wt {\mathbf D})\, h\left(Q_\ve(\wt{\mathbf \Delta})\right)\, Q_\ve(\wt{\mathbf \Delta})^{-z}$ on $\wt M$.  Here   $\mathbf  D$ is a Dirac-type operator,  and consequently $\mathbf \Delta:=\mathbf  D^2$ a Laplace-type operator,   $Q_\ve(\mathbf \Delta)$ defined in (\ref{eq:Qeps}) for some $\ve>0$  is a smooth deformation of   $\mathbf \Delta$, $P$  is a polynomial,  and $h$ some measurable function  on a contour around the spectrum of $Q_\ve(\mathbf \Delta)$. \\
\item  This yields  our third main result, Theorem \ref{thm:liftedregtraces}, which compares the corresponding $Q_\ve(\mathbf \Delta)$-regularised trace of $P(\mathbf D)\, h\left(\mathbf \Delta\right)$ and the $Q_\ve(\wt {\mathbf \Delta})$-regularised trace  of $P(\wt{\mathbf D})\, h\left(Q_\ve(\wt{\mathbf \Delta})\right)$.  \begin{itemize}
\item In the  $\Z_2$-graded case and  for $P\equiv 1$,  $h\equiv 1$ this gives back  Atiyah's  $L^2$-index theorem
%${\rm ind}(D)= {\rm ind}\sa{_\Gamma}(\wt D)$ where $\wt D$ is the lift of the  Dirac operator $D$ 
(Corollary \ref{cor:indres2}).
\item In the non-graded case and for $P(x)=x$, $h(x)= x^{-\frac{1}{2}}$,   assuming both operators $D$  and its lift $\wt D$ to be invertible,  the above constructions    show that the  eta-invariant of the lifted Dirac operator 
differs from   the eta-invariant of the Dirac operator $D$ on the base manifold  by  an ordinary $\Gamma$-trace ${\rm Tr}_\Gamma( \wt A_0-\mathcal A)$  of the difference of two $\Gamma$-invariant operators, one of which $\wt A_0$, is the lift of an $\ve$-local operator $A_0\in [A]_{\rm diag}$ (Corollary \ref{cor:eta}).
\end{itemize} 

Theorem \ref{thm:geometricop} discusses the case of geometric operators showing that the lifted $\zeta$-functions correspond to  integrals of densities generated by Pontrjagin forms on the fundamental domain and Chern forms on the auxillary bundle.\\
\end{itemize}
One  advantage of our approach  is  that it yields the $L^2$-Atiyah  theorem as an instance of the much more general {\it lifted trace defect formulae}. Here is the general scheme of the argument. Theorem \ref{thm:KVPScov} gives the lifted trace defect formulae.  Theorem  \ref{thm:liftedregtraces} compares $\zeta$-regularised traces of  operators with the $\zeta$-regularised traces of their lifted counterparts using the locality property of the only two {\it local linear forms} characterised in Theorem   \ref{thm:uniqueness}-- the canonical trace and the Wodzicki residue. Corollary \ref{cor:indres2} gives the $L^2$-index theorem combining the two previous ingredients.

The above arguments make use of  functions of pseudodifferential operators, in particular their complex power and the related logarithm. Even though the constructions are similar to the ones  for operators on closed manifolds, special care is to be taken in the open manifold case.  In  Section \ref{sec:PDOs}, we first review various  classes of pseudodifferential operators on  open manifolds, soon specialising to coverings. The essential difference between the various classes lies in the smoothing part, to which we therefore dedicate the first section --Section \ref{sec:smoothing}-- of the paper. We then relate pseudodifferential operators on coverings to pseudodifferential operators on the associated groupoid (Appendix \ref{sec:groupoids}) and the associated Hilbert module bundle (Appendix \ref{sec:appHM}), thereby  relating  constructions of complex powers and logarithms  on    groupoids and Hilbert module bundles with the ones on coverings presented here.   
 
 \bigskip
%\vspace{.8cm}
\newpage 
\tableofcontents

\setcounter{section}{0}

\vspace{-1.3cm}

\allowdisplaybreaks
  
\vfill \eject \noindent
 \section{$\Gamma$-invariant operators on coverings and functional calculus}

 \subsection{Operators with smooth kernels}
\label{sec:smoothing}
  Let $X$ be an $n$-dimensional manifold and $F\to X$ a vector bundle over $X$ of rank $k$.
To a linear operator $A\colon C^\infty_c(X, F)\to C^\infty(X, F)$ we assign its Schwartz kernel denoted by $K_A$, which is a distributional section of the bundle $F\boxtimes F$   over $M\times M$. 
The support of $A$ is the smallest subset of $X\times X$ on the complement of which $K_A$ vanishes as a distribution.
\begin{defn} Let $\Psi^{-\infty}(X,F)$ be the space of  linear operators   $A\colon C^\infty_c(X, F)\to C^\infty(X, F)$ with smooth Schwartz  kernel.
\end{defn}

Sobolev spaces will be useful to introduce another class of operators;  in order to have Sobolev spaces at hand, we henceforth  assume that our manifold $X$ has \emph{bounded geometry} \cite{Shu, MS, Kor}, a property verified by covering spaces of interest to us.
\subsubsection{Smoothing operators on manifolds of bounded geometry}\begin{defn}  
A Riemannian manifold $(X,g)$ is said to have \emph{bounded geometry} if
\begin{itemize}
\item it has positive injectivity radius (there is $r>0$ s.t. the exponential map is a diffeomorphism on $B(0,r)\subset T_x X$, $\forall x\in X$);
\item every covariant derivative of the Riemannian curvature tensor is bounded.
\end{itemize}
In the same way, a Hermitian vector bundle $F\to X $ has bounded geometry if every covariant derivative of the curvature is bounded.
\end{defn}
\begin{ex}
 Lie groups or homogeneous spaces with invariant metrics, compact Riemannian manifolds, regular $\Gamma$-covering of compact Riemannian manifolds    endowed with the induced Riemannian structure, all provide examples of manifolds of bounded geometry.
\end{ex}

We now assume the bundle $F\to X$ to be  of bounded geometry. The fundamental property is the existence of a \lq\lq good" partition of unity, which  allows to define Sobolev spaces  $H^s(X,F)$, see for instance  \cite[Lemmas 1.3, 3.22, and (1.3)]{Shu} and \cite[Definition 3.23]{Sch}.

\begin{rk}
 The Banach space structure of $H^s(X, F)$ is independent of the choices in the definition (see \cite[Lemma 3.24]{Shu}).
 Just as in the   case of closed manifolds,  Sobolev spaces can alternatively be defined by means of a (uniformly) elliptic operator, see \cite[Lemma 4.29 and Corollary 4.30]{Sch} for the comparison with this definition. 
\end{rk}

Let $H^\infty(X,F)=\cap_{k\in \N} H^k(X,F)$ denote the projective limit of $H^{k}(X,F), k\in \N$ and let  $H_\iota^{-\infty} (X,F)\supset \cup_{k\in \N} H^{-k}(X,F)$ denote the regular inductive limit of $H^{-k}(X,F),k\in \N$. Then  define
 (see e.g. \cite[Definition 5.3]{Roe2} and \cite[Lemma 2.13]{E}
\begin{eqnarray}\label{def:Ssmooth}
{\mathcal S}\Psi^{-\infty}(X,F)&:=&{\mathcal L}\left(H_\iota^{-\infty}(X,F),H^\infty(X,F)\right)\\ \nonumber&=&\cap_{(k,l)\in \N^2} {\mathcal L}\left(H^{-k}(X,F),H^l(X,F)\right)\\ \nonumber
& =&\cap_{(s,t)\in \R^2} {\mathcal L}\left(H^s(X,F),H^t(X,F)\right),
\end{eqnarray}
 where ${\mathcal L}(A,B) $ stands for continuous linear operators from a topological space $A$ to a  topological space $B$.
The notation we chose is inspired by Shubin, who calls these operators  $\mathcal S$-smoothing \cite[Def. 1, Ch. 3]{ShuG}.
By  \cite[Theorem 3.5]{Va},  an operator which smoothens sections has a smooth kernel, which leads to the following inclusion
  \begin{equation}
  \label{eq:Vaillant}
  {\mathcal S}\Psi^{-\infty}(X,F) \subseteq \Psi^{-\infty}(X,F).
  \end{equation}
 Following \cite{Shu2} and \cite{Roe2}, we set the following definition.
\begin{defn}
\label{def:Uinfty} Let $\Upsi^{-\infty}(X,F)$ denote the space of linear operators $A:  C^\infty_c(X,F) \to \Ci(X,F)$ with smooth kernel $K_A$ satisfying the following uniform boundedness condition:  for any multiindices $\alpha, \beta$    
$$\Vert \partial_x^\alpha \partial_y^\beta K_A(x,y)\Vert\leq C_{\alpha,\beta}\quad \forall (x,y)\in X\times X$$  for some positive constant $C_{\alpha,\beta}$.  
 \end{defn} 
Proposition 2.9 in \cite{Roe2} yields a refinement of  (\ref{eq:Vaillant}), namely \begin{equation}\label{eq:Roe}{\mathcal S}\Psi^{-\infty}(X,F)\subseteq \Upsi^{-\infty}(X,F).\end{equation}
 The uniformity follows from uniform estimates that naturally arise in the context of bounded geometry as they do for closed manifolds.
\begin{rk}If $X$ is a closed manifold, the three above spaces coincide:
$${\mathcal S}\Psi^{-\infty}(X,F)=\Upsi^{-\infty}(X,F)=\Psi^{-\infty}(X,F),
$$
since the equality in \eqref{eq:Vaillant} holds. Indeed any linear operator  $A\colon C^\infty_c(X, F)\to C^\infty(X, F)$ with smooth Schwarz kernel $K_A$ is smoothing when $X$ is closed as can be seen on direct inspection from the formula $(Au)(x)=\int_X K_A(x,y)u(y)dy$. 
\end{rk}

 The notion of properly  supported %, resp. $C$-local, 
 operators recalled in Definition \ref{defn:properly supported}, extends in a straightforward manner to linear operators $A:C^\infty_c(X, F)\to C^\infty(X,F)$.
The following definition is inspired by \cite{Shu}, and  follows the terminology of \cite{E}.
\begin{defn}   An operator $A\colon C^\infty_c(X,F)\to C^\infty( X,F)$   with Schwartz kernel $K_A$  has  {\bf finite propagation} if it is $C$-local (see Definition \ref{defn:Clocal}) for some positive $C$, i.e. if there is some  $C> 0$ such that $K_A(x,y)=0$ 
$\forall x, y$ with $\vert x-y\vert>C$ or equivalently, if  
$$\forall  u\in C^\infty_c(U,\C^k); \quad \supp (Au) \subset \{x: d(x, \supp u) \leq C\}.$$
Note that  requiring finite propagation is more constraining than the assumption of  quasi-locality of \cite{Roe2}.  
\end{defn}
Let $\Upsi_{\rm fp}^{-\infty}(X,F)$ denote the subspace of $\Upsi^{-\infty}(X,F)$ consisting of finite propagation operators with uniformly bounded smooth kernels. We have \cite{Shu2}  
\begin{equation}\label{eq:Shu} \Upsi_{{\rm fp}}^{-\infty}(X,F)\subsetneq {\mathcal S}\Psi^{-\infty}(X,F)\subseteq \Upsi^{-\infty}(X,F)\subseteq \Psi^{-\infty}(X,F).\end{equation}

 \begin{rk} The class $ \Upsi_{{\rm fp}}^{-\infty}(X,F)$ is strictly contained in ${\mathcal S}\Psi^{-\infty}(X,F)$ for the  heat operator $e^{-tD^2}$ on a bounded geometry manifold,   belongs to the class ${\mathcal S}\Psi^{-\infty}(X,F)$ (see for example \cite[\textsection 3.2]{Va})  but it does not have finite propagation.
\end{rk}

\begin{rk}\label{rk:nocomp}
 
 Whereas ${\mathcal S}\Psi^{-\infty}(X,F)$ is an algebra, the class $\Psi^{-\infty}(X,F)$ is not.   
 Indeed, the composition of two operators  with smooth kernels is defined only under appropriate decay conditions at infinity and  when this is the case, the Schwartz kernel of the composition might not be smooth. 
\end{rk}

 \subsubsection{ 
Coverings and classes of $\Gamma$-invariant operators with smooth kernel}

 \label{sec:smoothG}
 Let us now specialise to covering manifolds.
 Let $M$ be a (connected) closed manifold and $\wt M$ a regular covering  given by a $\Gamma$-principal bundle $\pi:\wt M\to M$ with $\Gamma = {\rm Aut}(p)$ the  discrete Lie group of deck transformations (smooth diffeomorphisms $\phi: \wt  M\to \wt M$ such that   $\pi\circ f= \pi$). 
Let  $\wt M$ be the universal cover so that $\Gamma=\pi_1(M)$ is the fundamental group of $M$.
 \begin{ex}$\R^n$ is a universal cover of $\T^n$ with group $\Gamma=\pi_1(\T^n)=\Z^n$.
 \end{ex}
If $M$ is a Riemannian manifold, we endow the covering $\wt M$ with the Riemannian structure induced by $\pi$. Let $E\to M$ be a Hermitian vector bundle and $\wt E:=\pi^* E\to \wt M$   its pullback.  
Both $\wt M$ and $\wt E$  are of bounded geometry \cite{Shu}. 

The action of $\Gamma$ on $\wt M$ via   diffeomorphisms $L_\gamma$
\begin{eqnarray*}
 \Gamma\times\wt M&\longrightarrow & \wt M\\
(\gamma, x)&\longmapsto &L_\gamma( x)
\end{eqnarray*}
 induces an action on  linear operators $A:C^\infty_c(\wt M,\wt E)\to C^\infty(\wt M,\wt E)$:
 $$L_\gamma^\sharp A:= L_\gamma \circ  A\circ L_\gamma^{-1}.$$
  The operator $A$  is said to be {\bf  $\Gamma$-invariant} whenever \begin{equation}
  \label{invariance}  
   L_\gamma ^\sharp A =  A \quad\forall \gamma \in \Gamma. 
 \end{equation}  
 The action $L_\gamma^\sharp$ stabilises $\Psi^{-\infty} (\wt M,\wt E)$. 
Imposing a $\Gamma$-invariance condition leads to the following subclass of operators.
\begin{defn}
\label{def:gammaclasses}
Let  $\Psi_\Gamma^{-\infty}(\wt M, \wt E)$, resp. $\Upsi_\Gamma^{-\infty}(\wt M, \wt E)$, resp. $\Upsi_{{\rm fp},\Gamma}^{-\infty}(\wt M, \wt E)$, resp. ${\mathcal S}\Psi_\Gamma^{-\infty}(\wt M, \wt E)$  denote the space of  $\Gamma$-invariant operators in $\Psi^{-\infty}(\tilde M, \tilde E)$, resp. $\Upsi^{-\infty}(\wt M, \wt E)$,  $\Upsi_{\rm fp}^{-\infty}(\wt M, \wt E)$, ${\mathcal S}\Psi^{-\infty}(X,F)$. 
\end{defn}
The following inclusions follow from (\ref{eq:Shu})
\begin{equation}\label{eq:ShuGamma} \Upsi_{{\rm fp}, \Gamma}^{-\infty} (\wt M, \wt E)\subsetneq {\mathcal S}\Psi_\Gamma^{-\infty}(\wt M, \wt E)\subseteq \Upsi_\Gamma^{-\infty}(\wt M, \wt E)\subseteq \Psi_\Gamma^{-\infty}(\wt M, \wt E).\end{equation}

\subsection{Classes of pseudodifferential operators}
\label{sec:PDOs}
In this section we discuss different classes of pseudodifferential operators on an open manifold  of bounded geometry (we also consider very general classes,  which possibly do not form algebras). 
 This short  review  which brings together and compares different approaches, follows \cite{Shu, MS, Kor, Roe2, E}.

 \subsubsection{Classical pseudodifferential operators on manifolds with bounded geometry}
 \label{sec:clpsibd}

 Let $X$ be an $n$-dimensional manifold and $F\to X$ a vector bundle over $X$ of rank $k$. We assume $X$ and $F$ are both of bounded geometry. 
 
\begin{defnlem} 
\label{defn:modulosmoothkernels}
 The following relations 
 \begin{itemize}
 \item
$A\sim B\Longleftrightarrow A-B \;\text{ has a smooth kernel}$
\item
$A\underset{ \rm \small diag}{\sim} B\Longleftrightarrow A-B \quad\text{has a smooth kernel supported outside the diagonal} $
 \end{itemize}
define   equivalence relations on the space ${\mathcal L}\left(C^\infty_c(X, F),C^\infty(X,F)\right)$ of linear operators acting on the space $C^\infty_c(X, F)$ of compactly supported sections of $F$ with values in the space $C^\infty(X,F)$ of smooth sections of $F$.  We write
  $[ A]$ (resp. $[ A]_{\rm diag}$) for the equivalence class of $A$ with respect to $\sim$ (resp. $\underset{ \rm \small diag}{\sim}$). 
\end{defnlem}

\begin{rk}\label{rk:simDelta}
\begin{itemize}
\item Clearly, we  have $A\underset{ \rm \small diag}{\sim} B\Longrightarrow A\sim B$.
\item Whereas the  equivalence relation $\sim$ is stable under composition of operators (whenever composable), the equivalence relation $\underset{ \rm \small diag}{\sim}$ is not. Indeed, if $A-A_1 =:R$ and $B-B_1=: S$ have smooth kernels supported outside the diagonal, then $AB= A_1S+ RB_1+ RS$ has a smooth kernel  but it might not be supported outside the diagonal since the supports of $A_1$ and $B_1$ intersect the diagonal.
\end{itemize}
\end{rk}

\begin{ex} 
 Given a linear  operator  $A:C^\infty_c(X, F)\to C^\infty(X,F),$ any  localisation  $\chi_1\, A\chi_2$    induced by two smooth functions $\chi_1$, $\chi_2$ whose compact supports  have a non void intersection in a trivialising set,  is properly supported.
 \end{ex}
 
 We shall make use of the  the existence of a \lq\lq good" partition of unity, \cite[Lemmas 1.3,  3.22]{Shu} and \cite[A1.1]{Shu} built as follows. For small enough $\rho$ (smaller than a third of the injectivity radius), there is a countable covering of $X$ by balls $B(x_i,\rho)$ centered at $x_i\in X$ with radius $\rho$ such that $d(x_i,x_j)\geq \rho$ for $i\neq j$ and any point $x\in X$ lies in at most $C_x$ such balls for some constant $C_x$. Moreover, there is a partition of   unity
$1=\sum_i\chi_i$ with smooth functions $\chi_i$ whose supports  $\supp \chi_i$ lie  in $  B(x_i,2\rho)$ and  which together with their derivatives taken in normal coordinates, are bounded independently of $i$. 

\begin{lem}\label{lem:A0S} Given a  linear operator $A:C^\infty_c(X, F)\to C^\infty(X,F)$  there is a properly supported operator $A_0:C^\infty_c(X, F)\to C_c^\infty(X,F)$ of finite propagation such that \begin{equation}\label{eq:PDOBG}A\underset{ \rm \small diag}{\sim} A_0.
\end{equation}  
For any $\ve>0$, the operator $A_0$ can be chosen $\ve$-local.
\end{lem}
\begin{proof}
Given a "good" finite  open cover  $\mathcal U=(U_i)_{i\in I}$  of $X$ and a   "good"  partition of unity $(\chi_i)_{i\in I}$  subordinated to $\mathcal U$, we  write the operator $A$ as
\begin{equation}
\label{Aij}
A=\sum_{i, j}\chi_iA\chi_j=\underbrace{\sum_{\supp \chi_i\cap\supp\chi_j\neq \emptyset}\chi_i A\chi_j}_{=:\sum_{\{i,j\}\in \mathcal P} A_{ij}}+\sum_{\supp \chi_i\cap\supp\chi_j= \emptyset}\chi_iA\chi_j
\end{equation}
where $\mathcal P$ is the set of  pairs $\{i,j\}$ satisfying $\supp \chi_i\cap\supp\chi_j\neq \emptyset$, and $A_{ij}:=\chi_i A\chi_j$. Then, $A_0:= \sum_{\{i,j\}\in \mathcal P} A_{ij}$ is a  properly supported operator of finite propagation since each $\chi_i A\chi_j$ is supported in balls with uniformly bounded radii. Moreover, by construction $S(A):=\sum_{\{i,j\}\in \complement\mathcal P}\chi_iA\chi_j$  has  Schwartz kernel supported outside the diagonal. \\
For $  \ve>0$, we can choose the diameter of the partition such that $\forall i\in I$,  $\diam U_i<\frac{\ve}{2}$, in which case $A_0$ is an $\ve$-local operator.  
 \end{proof} 
 \begin{defn}Given any real (resp. complex) number $m$, a linear operator $A\colon C^\infty_c(X, F)\to C^\infty(X,F)$ is   a  (resp. classical)  {\bf pseudodifferential operator} of order $m$  if  there is a properly supported operator $A_0\colon C^\infty_c(X, F)\to C_c^\infty(X,F)$ of finite propagation --- for any $\ve>0$, the operator $A_0$ can be chosen $\ve$-local--- with $A\underset{ \rm \small diag}{\sim} A_0$ as in (\ref{eq:PDOBG}) and  
 such that 
\begin{itemize} 
\item  the operator $S(A):= A-A_0$ lies in $ \Psi^{-\infty}(X,F)$,
\item the operator  $A_0$  is a sum $A_0=\sum_\alpha {\rm Op}(\sigma_\alpha)$ ( When applying the operator to a compactly supported section, the sum becomes finite, due to the local finiteness of a "good" open cover.) of   (classical) properly supported pseudodifferential operators ${\rm Op}(\sigma_\alpha)$ of order $m$ supported on  "good" open subsets of $X$ and identified via the   trivialising charts with pseudodifferential operators on open subsets of $\R^n$. The symbol  $\sigma_\alpha$ is interpreted as the symbol $\sigma(A)$  of $A$ seen in  the trivialising chart indexed by $\alpha$.
\end{itemize}
By abuse of notation we shall set \begin{equation}\label{eq:abusenotation}{\rm Op}(\sigma(A)):=A_0=\sum_\alpha {\rm Op}(\sigma_\alpha),\end{equation} so that $A\underset{ \rm \small diag}{\sim} {\rm Op}(\sigma(A))$.
 Let  $\Psi^m(X, F)$ (resp.  $\Psi_{\rm cl}^m(X, F)$) denote the class of such operators.
 \end{defn} 
\begin{rk}
\begin{itemize}
\item  
  Neither the class $\Psi(X,F):=\cup_{m\in \R} \Psi^m (X,F)$   nor  \\$\Psi_{\rm cl}(X,F):=\cup_{m\in \C} \Psi^m_{\rm cl }(X,F)$  form an algebra since two such operators do not generally compose,   compare Remark  \ref{rk:nocomp}.
\item 
The equivalence relations $\sim$ and $\underset{ \rm \small diag}{\sim}$ induce equivalence relations on   $\Psi_{\rm cl}^m(X,F)$ for any $m\in \C$, which preserve the symbol in any trivialising chart. 
\end{itemize}
\end{rk}

 \subsubsection{Uniform classical pseudodifferential operators}
 \label{subsec:ups}
We now specialise to the smaller class of     {\bf uniform} pseudodifferential operators, introduced by Shubin and Meladze  on Lie groups in \cite{MS} and by Kordyukov \cite{Kor} in the general setting of a bounded geometry manifold (see for instance \cite[Section 3]{Shu2}).  As we shall see later, it is  an appropriate class to host  pseudodifferential operators on coverings and consists of the usual H\"ormander properly supported pseudo-differential operators  with additional  uniformity conditions.
\begin{defn} 
\label{def:upsi}
Given (resp. $m\in \C$) $m\in \R$, let $ \Upsi^{m}(X, F)  $  (resp. $ \Upsi_{\rm cl}^{m}(X, F)  $) be the class of all     {\bf uniform   (resp. classical) pseudodifferential operators of order $m$}   i.e.,  operators 
$A\in \Psi^m(X,F)$  (resp. $A\in \Psi_{\rm cl}(X,F)$) 
which in a  "good" trivialising covering
$X=\cup_i B(x_i, \rho)$ of $X$ read $A\underset{ \rm \small diag}{\sim} A_0 $  as in  (\ref{eq:PDOBG}), where 
\begin{itemize}
\item  the operator $S(A):=A-A_0$ lies in $ \Upsi^{-\infty}(X,F)$
\item  and $A$ has a  {\bf   uniformly} bounded  symbol $\sigma(A)=\sigma(A_0)$   i.e., for any multiindices $\alpha, \beta$  there  is a constant $C_{\alpha,\beta}$  independent of $i$ such that 
\begin{equation}
\label{eq:uni-pdo}\Vert \partial_x^\alpha \partial_\xi^\beta \sigma(A) (x,\xi)\Vert\leq C_{\alpha,\beta}\, (1+\vert \xi\vert)^{m-\vert \beta\vert}\quad \forall (x,\xi)\in T^*B(x_i,\rho), 
\end{equation}
\item $\left(\right.$resp. for classical operators, with an additional {\bf  uniform} bound on the remainder terms in (\ref{eq:classical}), namely
 for any multiindices $\alpha, \beta$, for any $N\in \N$,  and for any excision function $\chi$ around zero, there  is a constant $C_{\alpha,\beta, N}$ independent of $i$ such that 
\begin{equation}\label{eq:uni-clpdo}\left. \left\Vert \partial_x^\alpha \partial_\xi^\beta \left(\sigma(A) (x,\xi)-\sum_{j=0}^{N-1} \sigma_{m-j}(A)\chi\right) (x,\xi)\right\Vert\leq C_{\alpha,\beta,N}\, (1+\vert \xi\vert)^{m-\vert \beta\vert-N}\quad \forall (x,\xi)\in T^* B(x_i,\rho) \right).\end{equation}
\end{itemize} 
On the grounds of (\ref{eq:Vaillant}) we can furthermore require that  $S(A)$ lies in ${\mathcal S}\Psi^{-\infty}(X,F)  $ which  defines  the following subclasses of operators:
\begin{equation}\label{eq:Vaillant1} {\mathcal S}\Upsi^m(X,F)\subseteq \Upsi^m(X,F)\quad \forall m\in \R, \quad {\mathcal S}\Upsi_{{\rm cl}}^m(X,F)\subseteq \Upsi_{\rm cl}^m(X,F)\quad \forall m\in \C . \end{equation}

\end{defn}  
\begin{rk}As can be seen from \eqref{eq:A1estimate}, for a vector bundle $E\to M$ on  a closed  manifold $M$, (\ref{eq:uni-pdo}) is verified by any pseudodifferential operator so that 
 $ \Upsi^{m}(M, E)  = \Psi^{m}(M, E)  $. Similarly, (\ref{eq:uni-clpdo}) is satisfied by any classical pseudodifferential operator and we have $ \Upsi_{\rm cl}^{m}(M, E)  = \Psi_{\rm cl}^{m}(M, E)  $.
\end{rk}
  
 %\sy{CAUTION: The following is proved in Shubin under the additional assumption that   the operator is $C$-local. This needs to be double checked.}

Similarly to pseudodifferential operators on closed manifolds, uniform  pseudodifferential operators modify the degree of regularity of   Sobolev  spaces by the order of the operator \cite[Remark (c) after Def. 3.3]{Shu}.
\begin{lem}\label{lem:ShuThm1} {} \cite[Proposition 2.20]{E}  An operator $A\in \Upsi^m(X,F)$  extends to a bounded operator 
\begin{equation}\label{eq:regm}
\overline A: H^s(X,F)\to H^{s-m}(X,F) \;, \text{ for any  } s\in \R.
\end{equation}
\end{lem}
\begin{rk}\begin{itemize}
\item  The proof of \cite[Proposition 2.20]{E}, stated for quasi-local uniform pseudodifferential operators, relies on the local finiteness of the covering and does not use   quasi-locality. Hence the proof extends to elements of $\Upsi^m(X,F)$.
\item Consequently, the space of uniform  pseudodifferential operators of  real order $m$ compares with the space $\mathcal O p ^m(X, F)$ used in \cite[pag. 11]{Va} to denote the space of all \lq\lq $m$-regularising operators\rq\rq, i.e. the linear operators $A\colon C_c^\infty(X, F)\to C^\infty_c(X, F)'$ which extend as in \eqref{eq:regm}
$$
{\mathcal S}\Upsi^m(X, F)\subsetneq \Upsi^m(X, F)\subsetneq  \mathcal O p ^m(X, F)\ .
$$
\end{itemize}
\end{rk}
This leads to the following identifications.
\begin{prop}\label{prop:Comparison} (compare with \cite[Lemma 2.22]{E})
$${\mathcal S}\Psi^{-\infty}(X,F)=  \cap_{m\in \R} \Upsi^m(X,F)=\Upsi^{-\infty}(X,F).$$
Consequently, ${\mathcal S}\Upsi^m(X,F)= \Upsi^m(X,F)$ for any real number $m$  and 
 ${\mathcal S}\Upsi_{\rm cl}^m(X,F)= \Upsi_{\rm cl}^m(X,F)$ for any complex number $m$.
\end{prop}
\begin{proof}For any real number $m$, on the one hand we have
$\cap_{m\in \R}\Upsi^m(X,F)\subset {\mathcal S}\Psi^{-\infty}(X,F)$ as a consequence of \eqref{eq:regm}. On the other hand,   we know by (\ref{eq:Roe}) that ${\mathcal S}\Psi^{-\infty}(X,F)\subset \Upsi^{-\infty}(X,F)\subset\Upsi^m(X,F)$ for any real number $m$. Hence ${\mathcal S}\Psi^{-\infty}(X,F)\subset  \cap_{m\in \R} \Upsi^m(X,F)$ and the first identity follows.\\
As for the second identity, we have the straightforward inclusion $\Upsi^{-\infty}(X,F)\subset \Upsi^m(X,F)$ for any real number $m$  which yields the inclusion from right to left. The inclusion from left to right follows from observing that the uniform estimates \eqref{eq:uni-pdo} imply the uniform boundedness of the derivatives of the Schwartz kernel of the operator.
\end{proof}

\subsection{$\Gamma$-invariant classical pseudodifferential operators on covering spaces}
\label{sec:cover}

Let $M$ be a (connected) closed manifold and $\pi\colon \wt M\to M$ a regular $\Gamma$-covering as in Section \ref{sec:smoothG}.

\begin{defn}
\label{def:gammaclasses2}
Imposing the $\Gamma$-invariance condition \eqref{invariance} leads to the following subclasses
 $\Psi_\Gamma^m(\wt M,\wt E)$, $\Upsi_\Gamma^m(\wt M,\wt E)$, $ \Psi_\Gamma^{-\infty}(\wt M,\wt E)$, $\Upsi_\Gamma^{-\infty}(\wt M,\wt E)$  of  $\Gamma$-invariant operators in the corresponding classes 
   defined in Section \ref{sec:clpsibd}. The spaces $\Psi_{  {\rm cl},\Gamma}^m(\wt M,\wt E)$ and $\Upsi_{  {\rm cl}, \Gamma}^m(\wt M,\wt E)$ are defined analogously. 
\end{defn}

\begin{rk} 
\label{rk:Gamma-sup} As consequences of cocompactness of  $\wt M$, 
  a $\Gamma$-invariant operator on $\wt M$ is properly supported if and only if its Schwartz kernel has compact support in $(\wt M\times \wt M)/\Gamma$, see \cite[Chapter 3, above Definition 3]{ShuG}.
 \end{rk}

We   equip   $\wt M$  with a $\Gamma$-invariant locally finite open cover  in the following way:  given a finite open cover $\mathcal U_M=\{U_j, j=1, \cdots, N\}$  of $M$, we lift it to $\wt M$ and  take all the connected components to have a cover by connected open subsets. We obtain a $\Gamma$-invariant locally finite open cover
$
\wt M=\bigcup_{{j=1,..,N \atop \gamma\in \Gamma}} \gamma\, U_j.
$
We then build  a $\Gamma$-invariant partition of unity  \begin{equation}\label{eq:tildechi}\wt\chi_j:=\{\chi_{j, \gamma}\in C^\infty_c(\gamma\, U_j), \gamma\in\Gamma\}_{ j=1,\cdots, N}\end{equation}  subordinated to this cover with $\chi_{j, \gamma}(x)=\chi_{j, e}(\gamma^{-1}x)$. This way,  a partition of unity $\{\chi_j\}_{j=1,\cdots, N}$  of $M$ subordinated to the covering $\mathcal U_M$ is lifted to a $\Gamma$-invariant partition of unity $\{\wt \chi_j\}_{j=1,\cdots, N}$, which is a "good" partition of unity in the sense of manifolds with bounded geometry.
 
Such a partition of unity combined with a subordinated trivialisation of $\wt E$  can  be used to construct   Sobolev spaces $H^s(\wt M,\wt E)$ of sections on $\wt E$ \cite[Definition 1, \textsection 3.9]{Sch}.
As a consequence of the corresponding property on manifolds with bounded geometry, see Lemma \ref{lem:ShuThm1}, we have:
\begin{rk}
\begin{itemize}
\item    An  operator $A\in \Upsi^m_\G(\wt M,\wt E)$  extends to a bounded operator\\ $H^s(\wt M,\wt E)\to H^{s-m}(\wt M,\wt E) $, for any $ s\in \R$. 
\item Consequently, by Proposition \ref{prop:Comparison}, $\Upsi_\G^{-\infty}(\wt M,\wt E)=\cap_{m\in \R}\Upsi^m_\G(\wt M,\wt E)$. 
\end{itemize}
\end{rk}

\subsection{Lifted operators}
As proved in \cite{ShuG}, $\ve$-local pseudodifferential operators can be lifted from $M$ to $\wt M$, and their lifts are uniform properly supported operators. 
We include  the proof of this classical fact for completeness.  
 \begin{lem}
 {}\cite[Proposition 1, \textsection 3.9]{ShuG}
 \label{prop:ep-local}
 With the  notation as above,  let $r_0:=\inf_{x\in X} \{d(x, \gamma x), \gamma\in \Gamma\setminus \{e\}\}>0$, where $e$ is the unit of $\,\G$ and  let  $A\colon C^\infty(  M, E)\rightarrow C^\infty (M,  E)$  be an $\ve$-local operator  with $\ve<\frac{r_0}{2}$. 
\begin{enumerate}
\item 
 There exists a unique $\ve$-local operator $\wt A\colon C_c^\infty (\wt M,\wt E)\to C^\infty (\wt M,\wt E)$ such that for any lifted local section $\wt s $ of $F$ of a local section $s$ of $E$ \begin{equation}\label{tildepistar}\widetilde A\,  \tilde s=\pi^*(As). \end{equation} 
 With the notations of the introduction, we write $\pi_\sharp (\widetilde A)= A; \quad \pi^\sharp ( A)= \widetilde A.$
\item If moreover $A $ lies in $ \Psi^m(M, E) $ for some $m\in \R$ (resp. $\,\Psi_{cl}^m(M, E)$ for some $m\in \C$), we have (with the notation of \eqref{eq:simsymb}, see Appendix A) 
 \begin{equation}
 \label{eq:liftedsymbolm}
 \sigma(\wt A) = \wt {\sigma (A)}, \quad \left({\rm resp.}\quad  \sigma_{m-j} (\wt A) = \reallywidetilde {\sigma_{m-j}(A)},\quad \forall j\geq 0\right). 
 \end{equation}   This is
 to be understood as a local identity in appropriate local trivialisations around a point $ x$.
In particular,  $\wt A$ lies in $ \Upsi_\Gamma^m( \wt M,\wt E)$  (resp. $ \,\Upsi_{{\rm cl}, \Gamma}^m(\wt M,\wt E)$).
\end{enumerate}
\end{lem}
\begin{proof}
We have that if $d(x,y)<\ve$, then $d(\gamma x, y)>\ve$ for all $ e\neq \gamma\in \Gamma$. If $K_A$ denotes the Schwartz kernel of $A$, define $\wt A$ by constructing the operator with Schwartz kernel 
$$
K_{\tilde A}= \left\{\begin{array}{cc}K_A(\pi(x),\pi(y)), & d(x,y)<\ve\\ 0\;\;\;\;, & \text{elsewhere}.\end{array}\right.
$$
To show (2), let  $(V,\Phi)$ be a local trivialisation of $E$  where $V$ is an evenly covered open set. Recall that the symbol of $A$ on this local chart, denoted by $\sigma_V (A)(x, \xi)$, is by definition  the symbol of the operator $\Phi^\sharp A_V$ acting on  matrix valued functions on $\phi(V)\subset \R^n$, where $A_V$ is the localization of $A$ on $V$.  The symbol of the lifted operator $\wt A$ is described as follows. Let $\pi^{-1}(V)=\bigsqcup_{\gamma\in \Gamma}U_\gamma$;  on each local chart $(U_\gamma, \phi\circ \pi)$ the symbol is  defined as the symbol of $(\Phi\circ \pi)^\sharp \wt A_{U_\gamma}$. It follows that  $
\sigma_{U_\gamma}(\wt A)( x,\xi)=\sigma_V(A)(\pi( x),\xi)$. \\
The fact that  $\wt A$ lies in $ \Upsi_\Gamma^m( \wt M,\wt E)$  (resp. $ \,\Upsi_{{\rm cl}, \Gamma}^m(\wt M,\wt E)$) then follows from  the fact that $ \Upsi^m( M, E)= \Psi_\Gamma^m( \wt M,\wt E)$  (resp. $ \,\Upsi_{{\rm cl}}^m( M, E)$) = $\Psi_{{\rm cl, \Gamma}}^m( M, E)$).
\end{proof}

  We now combine Proposition \ref{prop:ep-local} with  the partition of the unity \eqref{eq:tildechi} to lift operators {\it modulo $\underset{ \rm \small diag}{\sim}$}, and have therefore a \lq\lq lifted" analogue of Proposition \ref{prop:PsidoU}.
\begin{prop}
\label{prop:ShubinTh1}  
Let $ \ve>0$ and  $\mathcal A$  be an operator in $\Psi_\Gamma(\wt M, \wt E)$.
\begin{enumerate}
	\item If $\mathcal A$ is $\ve$-local, with the notation of \eqref{tildepistar}, there exists a unique $A $  such that \begin{equation}\label{eq:tildeAtildes}\mathcal A=\pi^\sharp A.\end{equation}
\item In general, there exists an  $\ve$-local operator $A\in \Psi(M, E)$  
  such that \begin{equation}\label{eq:mathfrakAtildeA}\mathcal A\underset{ \rm \small diag}{\sim} \wt A.\end{equation}
  Consequently, the symbols of the two operators relate by
 \begin{equation}\label{eq:liftedsymbol}\sigma(\mathcal A)\sim \reallywidetilde{\sigma(A)},\end{equation}
  independently of the choice of $\wt A\in [\mathcal A ]_{\rm diag}$.
If $\mathcal A$ is $\ve$-local, in particular if it is a differential operator, then $\mathcal A =\wt {\pi_* \mathcal A}.$ 
\end{enumerate}
\end{prop}
\begin{proof} 
 Let $\{\wt \chi_j\}_{ j=1, \cdots, N}$ be a $\Gamma$-invariant partition of unity subordinated to a  cover  $X=\bigcup_{{j=1,..,N \atop \gamma\in \Gamma}} \gamma\, U_j
$ with open sets $U_j$ of diameter smaller than $\ve$. As in the proof of Lemma \ref{lem:A0S} we write a $\G$-invariant  operator $\mathcal A\in \Psi_\Gamma (\wt M, \wt E)$ %(resp. $ \Upsi_\Gamma(\wt M, \wt E)$)
 as
$$
\mathcal A=\sum_{i, j}\wt \chi_i \mathcal A\wt \chi_j=\sum_{\supp \wt \chi_i\cap\supp\wt \chi_j\neq \emptyset}\wt \chi_i\mathcal A\wt \chi_j+\sum_{\supp \wt \chi_i\cap\supp \wt\chi_j= \emptyset}\wt \chi_i\mathcal A\wt \chi_j,
$$
where   with a slight abuse of notation,  using the notations of \eqref{eq:tildechi}, we have set $\supp \wt \chi_i= \cup_{\gamma\in \Gamma}\supp \wt \chi_{i,\gamma}$.\\
Choosing   the diameter of the partition small enough and applying Proposition \ref{prop:ep-local} to the  $\ve$-local operators $\chi_i \mathcal A\chi_j $, we have  $$\chi_i\, \mathcal A\,\chi_j =\reallywidetilde { \pi_\sharp(\widetilde \chi_i\, \mathcal A\, \widetilde \chi_j)}, $$ which yields
\begin{equation}
	\label{eq:localdescriptionA}
\mathcal A= \sum_{\supp\wt \chi_i\cap\supp\wt \chi_j\neq \emptyset}\wt{ A}+ S(\mathcal A)=\wt A+ S(\mathcal A)
\end{equation}
with 
  \begin{equation}
   A:= \sum_{\supp\wt \chi_i\cap\supp\wt \chi_j\neq \emptyset}\pi_\sharp\left(\chi_i \,  \mathcal A\,\chi_j\right)
\end{equation}
 and   
 $$
 S(\mathcal A):=\mathcal A- \wt A=\sum_{\supp \wt\chi_i\cap\wt \supp\chi_j= \emptyset}\wt \chi_i \,\mathcal A\,\wt \chi_j\ .
 $$
  a linear operator with smooth kernel supported outside the diagonal.  \\
If $\mathcal A$ is $\ve$-local, then the above construction reduces to 
 $$
 \mathcal A= \reallywidetilde{\sum_{\supp \wt \chi_i\cap\supp\wt \chi_j\neq \emptyset}\, \wt \chi_i (\pi_*\mathcal A)\, \wt \chi_j}= 
 \wt{\pi_*\mathcal A}. 
 $$
 This proves \eqref{eq:mathfrakAtildeA} from which \eqref{eq:liftedsymbol} then follows.
\end{proof}
 
\begin{rk}In view of \eqref{eq:liftedsymbol}, properties of pseudodifferential operators such as being classical, the order, invertibility of the principal symbol can be lifted without ambiguity. 
\end{rk}
On the grounds of the above Remark, we set the following
\begin{defn}\label{def:elliptic_frak} With the notations of Proposition \ref{prop:ShubinTh1},  
an operator $\mathcal A$ in $\Upsi_{{\rm cl},\Gamma}(\wt M, \wt E)$ is {\bf elliptic } whenever $A$ is elliptic, i.e. whenever its principal symbol is invertible.
\end{defn}

On the grounds of the above proposition, we set the following
\begin{defn}
\label{defn:classlift}
Let  $ A\in \Psi_{\rm cl}(M,E)$, let $A_0$ be $\ve$-local such that $A\underset{ \rm \small diag}{\sim} A_0$ as in \eqref{eq:PDOBG}. We define the lift of the class $[A]_{\rm diag}$ to  
\begin{equation}
\label{eq:liftedOp}
\reallywidetilde {[A]_{\rm diag}}:= {[\wt A_0]}_{\rm diag}\ .
\end{equation}
\end{defn}
With this defintion at hand, for any $\mathcal A\in \reallywidetilde {[A]_{\rm diag}}$ we have
\begin{equation}\label{eq:sigmafraktildeAzero}\sigma(\mathcal A)\sim \reallywidetilde {\sigma(A_0)}.\end{equation}

\subsection{Lifting functions of operators}
 Let $E$ be a hermitian vector bundle over the closed Riemannian manifold $M$.\\
We borrow the following definition from \cite{ALNP}.
\begin{defn}
\label{defn:weight} 
We  call a {\bf weight} in $\Psi_{\rm cl}(M, E)$, an operator $Q\in \Psi_{\rm cl}(M, E)$ such that
\begin{enumerate}
\item $Q$ is invertible, namely   its kernel is non trivial or equivalently,  it admits an inverse defined on $L^2( M, E)$,
\item $Q$ has positive order $q$,
\item $Q$ has a principal angle $\theta$, which means  that there exists a ray $R_\theta= \{re^{i\theta}, \; r\geq 0\}$,  called {\bf spectral cut} , which is disjoint from the spectrum of the ${\rm End}(E_x)$-valued leading symbol $\sigma_L(Q)(x,\xi)$ for any $x\in M$, $\xi \in T_x^*M \setminus \{0\}$.
\end{enumerate}
This last condition implies that the spectrum of the operator $Q$  lies outside a cone $\Lambda_\theta$ containing the ray $R_\theta$, \cite[Lemma 1.6]{ALNP}.
\end{defn} 
\begin{ex}Let  $D$ in $ \Psi^d_{\rm cl}(M, E)$ be an essentially self-adjoint elliptic differential operator of positive order $d$.  Then $\Delta:= D^2$ is a  non-negative elliptic differential operator on $M$ of positive order $q:=2d$ and the operator $\Delta+1$ is a differential operator which defines a weight.
\end{ex}
\begin{ex}\label{ex:Delta} With the same notations as in the above example, we can instead add  to $\Delta$ a smoothing operator $ \chi_{[-\ve,\ve]}(\Delta)$ with $\ve>0$ chosen small enough so that it  coincides with the orthogonal projection $\chi_0(\Delta)$ onto the kernel of $\Delta$.  Then \begin{equation}
\label{eq:Qeps}Q_\ve(\Delta):= \Delta+\chi_{[-\ve,\ve]}(\Delta)
\end{equation} defines a weight with spectral cut  $R_\pi=\R_{\leq 0}$  .
\end{ex}
A weight $Q\in \Psi_{\rm cl}(M,E)$ satisfies  a resolvent estimate, see \cite[(9.30)]{Shub} and \cite[Cor. 1, p. 298]{Se}:
\begin{equation}
\label{eq:estimateR}
\|(Q-\lambda )^{-1}\|_{s, s+l}\leq C_{s,l} \,\vert\lambda\vert^{-1+\frac{l}{q}} \;\;\forall \,0\leq  l\leq q\;\; \forall  \la\in \Lambda_\theta\cap\{|\la|>R>0\}.
\end{equation}  
Let $\Gamma_\theta$ be a contour around the ray $R_\theta$, then to a  measurable  function $h$ on $\Gamma_\theta$, such that $\vert h(\lambda)\vert \leq \vert \lambda\vert^{-\delta} $  for some positive $\delta$,   we  can associate the  operator  \begin{equation}\label{eq:hQ}h(Q) :=-\frac{1}{2i\pi} \int_{\Gamma_\theta} h(\lambda)\, (Q-\lambda)^{-1}\, d\lambda,\end{equation}
whose symbol is given by the corresponding Cauchy integral \begin{equation}\label{eq:sigmahQ}\sigma(h(Q))\sim h_\star (\sigma(Q)):=-\frac{1}{2i\pi}\, \int_{\Gamma_\theta} h(\lambda)\, (\sigma(Q)-\lambda)^{\star -1}\, d\lambda, \end{equation}
  where the exponent $\star k$ stands for the $k$-th $\star$-product exponent of symbols. We refer the reader to any book on pseudodifferential operatorsfor the precise definition of the $\star$-product, see e.g. \cite[(3.41)]{Shub}.
\begin{ex} 
For a polynomial  $h(x)=\sum_{k=0}^n a_k x^k$, \eqref{eq:hQ} yields the  operator  $h(Q):=\sum_{k=0}^n a_k \, Q^k$ with symbol  $h_\star(\sigma(Q))= \sum_{k=0}^n a_k \sigma(Q)^{\star k}$.
\end{ex}
\begin{ex} 
If $Q=D^2$, with $D$ an essentially self-adjoint operator, then the function $h(x)=\frac{1}{\sqrt{ x}}$ yields   the   operator $\vert D\vert^{-1}:= Q^{-\frac{1}{2}}$ from which we build the sign operator
\begin{equation}
\label{eq:sgn}
{\rm sgn}(D):= D\,h(Q)= D\, \vert D\vert^{-1}\quad \text{with symbol} \quad \sigma\left({\rm sgn}(D)\right)\sim h_\star (\sigma(D)):=\sigma(D)\star \left(\sigma(\Delta)\right)^{\star-\frac{1}{2}}.
\end{equation}
\end{ex}
 
 Definition \ref{defn:weight} carries out to  $\Upsi_\Gamma(\wt M, \wt E)$, up to the fact that, in contrast with the closed case (see \cite[Def. 3.6]{ALNP} for details), the spectrum being not necessarily purely discrete in the noncompact case, we need an extra condition for the existence of an Agmon angle, defined as follows.
 
\begin{defn}\label{def:agmonL2}  For an angle $\beta$ and for $\epsilon >0$, denote 
 $$V_{\beta,\epsilon}:=\{z\in \C\;: |z|<\epsilon \}\cup \{z\in \C\setminus 0\,:\; {\rm arg} z \in (\beta-\epsilon, \beta+\epsilon)\} .$$ 
    Then $\beta$ is called an {\bf Agmon angle} for $A\in \Upsi_{{\rm cl},\Gamma}(\wt M, \wt E )$ if there is some $ \epsilon >0$ such that $\spec( A)\cap V_{\beta, \epsilon}=\emptyset$
 \end{defn}

\begin{defn}
We  call a {\bf weight} in $\Upsi_\Gamma(\wt M, \wt E)$, an operator $\mathfrak Q\in \Upsi_\Gamma(\wt M, \wt E)$ such that
\begin{enumerate}
\item $ \mathfrak Q$ is invertible in the strong sense of the term, namely that it admits an inverse defined on $L^2(\wt M, \wt E)$,
\item $\mathfrak Q$ has positive order $q$,
\item $ \mathfrak Q$ has a principal angle $\theta$ as in Definition \ref{defn:weight}, 
\item $\theta$  is an Agmon angle for  $\mathfrak Q$.
\end{enumerate}
\end{defn}
 \begin{rk} \label{rk:sigmafrak}
In view of \eqref{eq:liftedsymbol}, for any weight $\mathfrak Q\in \Upsi_\Gamma(\wt M, \wt E)$, there exists an operator $Q$ in $\Psi( M,  E)$) such that
 \begin{equation}\label{eq:sigmafrak}\sigma(\mathfrak Q)\sim \reallywidetilde{\sigma( Q)}
 \end{equation} so that it has the same order and same principal angle. Moreover it can be chosen invertible modulo addition of the projection onto its kernel. Hence for any  weight $\mathfrak Q$    in $\Upsi_\Gamma(\wt M, \wt E)$ there is a weight  $ Q$ in $\Psi( M,  E)$ with the same spectral cut and 
such that \eqref{eq:sigmafrak} holds.
 \end{rk} 
\begin{figure}[ht]
	\label{fig1}
		\caption{Agmon angle $\beta$}
	\centering
\includegraphics[]{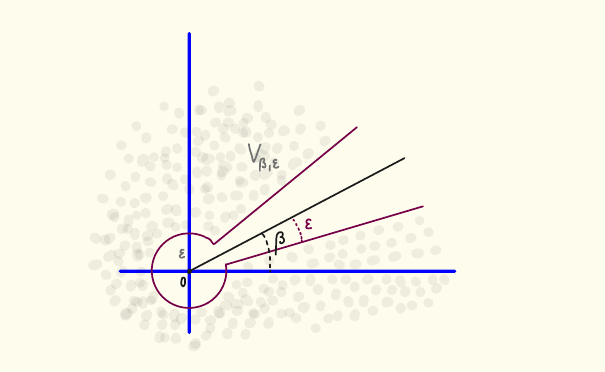}
\end{figure}

\begin{lem}\label{lem:hstarQ} Let $\mathfrak Q$  be a weight in $\Upsi_\Gamma(\wt M, \wt E)$ with spectral cut $R_\theta$ and let $ Q$  be a weight in $\Psi( M,  E)$ with the same spectral cut as in Remark \ref{rk:sigmafrak}.
With the notations of \eqref{eq:sigmafrak},  for every     measurable  function $h$ on a contour $\Gamma_\pi$ around the ray $R_\theta$,  such that
  $\vert h(\lambda)\vert\leq\vert  \lambda\vert^{-\delta} $  for some positive $\delta$, 
we have:
\begin{equation}\label{eq:hstarQ}
%\sigma(\reallywidetilde{h(Q)_0})\sim 
h_\star\left(\sigma(\mathfrak Q)\right)\sim\reallywidetilde{h_\star(\sigma(Q))},\end{equation}
where we have used the notation of (\ref{eq:sigmahQ}).
\end{lem}
\begin{proof}
The star product $\star$, which is a local operation for it  only involves derivatives,   commutes with the lift. For two local symbols $\sigma$ and $\tau$, we have
$\reallywidetilde{\sigma\star \tau}\sim \wt\sigma\star \wt \tau$, which implies $\reallywidetilde{(\sigma-\lambda)^{\star -1}}\sim (\wt \sigma-\lambda)^{\star -1}$ and hence
$$\reallywidetilde{ h_\star (\sigma )}\sim h_\star (\wt \sigma)=-\frac{1}{2i\pi}\, \int_\Gamma h(\lambda)\, (\wt \sigma-\lambda)^{\star -1}.$$Implementing $h_*$    therefore yields
$$\sigma(\mathfrak Q )\sim\reallywidetilde{\sigma(Q)}\Longrightarrow h_\star(\sigma(\mathfrak Q))\sim h_\star(\reallywidetilde{\sigma(Q)})\Longrightarrow h_\star(\sigma(\mathfrak Q))\sim \reallywidetilde{h_\star(\sigma(Q))} .$$
\end{proof}

\begin{ex} Let $Q$ be a weight with spectral cut $R_\theta$. 
For $\Re(z)>0$, the function $h(x)=x_\theta^{-z}$ with the complex power determined by the angle $\theta$,  yields the complex power
\begin{equation}\label{Qz}Q_\theta^{-z}:=-\frac{1}{2i\pi}\, \int_{\Gamma_\theta} \lambda_\theta^{-z}\, (Q-\lambda)^{\star -1}\, d\lambda,\end{equation}
which  can be extended to any complex value $z$ setting $Q_\theta^{-z}:=Q^k\, Q_\theta^{-z+k}$ for $\Re(z)>-k$.
\end{ex}

 With the same notations as in the Example \ref{ex:Delta}, the differential operator $\Delta$ lifts to a differential operator $\wt \Delta$ whose leading symbol is the lifted leading symbol of $\Delta$.

With the notation introduced in Appendix \ref{sec:appHM}, the isomorphism $\Phi$  induces a map
\begin{eqnarray*}
\Upsi_{{\rm cl},\Gamma} (\wt M, \wt E) &\longrightarrow &\Psi(M; E\otimes \mathcal H)\\
A&\longmapsto &\Phi^\sharp A.
\end{eqnarray*}
 Let $\Delta_{\mathcal H}:=\Phi^\sharp \wt \Delta$  be the corresponding elliptic operator in $\Psi_{\rm cl}(M, E\otimes {\mathcal H})$.
    \begin{prop}
\label{prop:chieps} For any positive $\ve$, the operator  \begin{equation}\label{eq:tildeQeps}Q_\ve(\wt \Delta):=\wt \Delta+\chi_{[-\ve,\ve]}(\wt \Delta),\end{equation} resp. $Q_\ve(D_{\mathcal H}):=\Delta_{\mathcal H}+\chi_{[-\ve,\ve]}(\Delta_{\mathcal H})$, defines a weight in  $\Upsi_{{\rm cl},\Gamma} (\wt M, \wt E)$, resp.  $\Psi_{\rm cl}(M, E\otimes \mathcal H)$  with spectral cut  $R_\pi=\R_{\leq 0}$   and we have
\begin{equation}\label{eq:QHtilde}Q_\ve(\Delta_{\mathcal H})=\Phi^\sharp Q_\ve(\wt \Delta).
\end{equation}
\end{prop}
\begin{proof} To ensure that the operator  $Q_\ve(\wt D)$, resp. $Q_\ve(D_{\mathcal H})$ defines a  weight (for the latter, see also \cite{ALNP}), we need to check that 
\begin{enumerate}
\item the operator  $\chi_{[-\ve,\ve]}(\wt \Delta)$, resp. $\chi_{[-\ve,\ve]}( \Delta_{\mathcal H})$ has a smooth Schwartz kernel in  $\Psi_\Gamma(\wt M, \wt E)$, resp.  $\Psi(M, E\otimes \mathcal H)$) which follows from   \cite[Cor 3.6]{Va} , resp. from \cite{BFKM}.
\item the operator  $\wt \Delta+\chi_{[-\ve,\ve]}(\wt \Delta)$, resp. $ \Delta_{\mathcal H}+\chi_{[-\ve,\ve]}( \Delta_{\mathcal H})$  is invertible, which is  an immediate consequence of its non-negativity.
\end{enumerate} 
The compatibility of the map $\Phi$ with functional calculus (see \ref{prop:dict}) implies  \eqref{eq:QHtilde} .
\end{proof}

Vassout's functional calculus on groupoids recalled in Appendix \ref{sec:groupoids},  allows to extend the notion of weight to the  groupoid  $G(M):=(\wt M\times \wt M)/\Gamma$ associated  with  $\Upsi_{{\rm cl},\Gamma} (\wt M, \wt E)$, by the  isomorphism $\rho$ defined in  \eqref{eq:G_x}. Indeed, the map  
\begin{eqnarray*}
 \Upsi_{{\rm cl},\Gamma} (\wt M, \wt E)&\longrightarrow& \Psi(G, E)\\
A&\longmapsto & \rho^\sharp A, 
\end{eqnarray*}
 induced preserves the properties 1)-3) of a weight and therefore transforms a weight $\mathfrak Q\in\Upsi_{{\rm cl},\Gamma} (\wt M, \wt E)$ with spectral cut $\theta$ to a weight $\rho^\sharp (\mathfrak Q)$ on the associated groupoid with the same spectral cut. The map $\rho^\sharp$ preserves the estimate \eqref{eq:estimateR} on weights, which enables us to transport the related functional calculus from the groupoid to $\Upsi_{{\rm cl},\Gamma} (\wt M, \wt E)$.  With the  notations of \eqref{eq:hQ}   and for a measurable  function $h$ on  on a contour $\Gamma_\pi$ around the ray $R_\pi=]-\infty, 0]$,  such that $h(\lambda)\leq \lambda^{-\delta} $  for some positive $\delta$, we can  define 
\begin{equation}\label{eq:rhohQ}h(\mathfrak Q)=\left(\rho^\sharp\right)^{-1} ( h(\rho^\sharp (\mathfrak Q)),
\end{equation} 
whose symbol is given by \begin{equation}\label{eq:sigmahfrakQ}\sigma(h(\mathfrak Q))\sim h_\star (\sigma(\mathfrak Q)):=-\frac{1}{2i\pi}\, \int_{\Gamma_\theta} h(\lambda)\, (\sigma(\mathfrak Q)-\lambda)^{\star -1}\, d\lambda.\end{equation}
\begin{ex}
For $\Re(z)>0$, the function $h(x)=x_\theta^{-z}$ with the complex power determined by the angle $\theta$,  yields the complex power
\begin{equation}\label{frakQz}\mathfrak Q_\theta^{-z}:=-\frac{1}{2i\pi}\, \int_{\Gamma_\theta} \lambda_\theta^{-z}\, (\mathfrak Q-\lambda)^{\star -1}\, d\lambda,\end{equation}
where $\mathfrak Q\in\Upsi_{{\rm cl},\Gamma} (\wt M, \wt E)$ is a weight with spectral cut $R_\theta$.
\end{ex}

\subsection{Lifting complex powers to coverings}
We now specialise to  $h:x\mapsto x_\theta^{-z}$, where $\theta$ stands for the determination of the complex power. Applied to a weight $Q\in \Psi_{\rm cl}(M,E)$ with spectral cut $R_\theta$, this gives rise to complex powers $Q_\theta^{-z} \in  \Psi_{\rm cl}(M,E)$.
\begin{rk}   Let $m\in \R$, $z\in \C$ and $Q\in \Psi^m_{\rm cl}(G, E)$ a positive (with respect to the $L^2$-inner product) elliptic, invertible operator. It defines a weight with spectral cut  $R_\pi=\R_{\leq 0}$  (we drop the mention of $\theta=\pi$ in the notation). The complex power  $Q^{-z}$
  is defined in \cite[p. 25]{Vas}  according to  \eqref{frakQz}
and proved to  belong to $ \Psi_{\rm cl}^{-mz} (G, E)$ and to act  as an an element of  $\mathcal L(H^{t-m\Re z}(\mathcal W), H^t(\mathcal W))$ for any $t\in \R$. By Proposition \ref{prop:Gvscov}, the inverse map $\rho^{-1} $ identifies   the operator $Q^{-z}$   with  the  operator ${\rho^\sharp}^{-1}(Q^{-z})$ in  $\Upsi_{\rm cl}^{-mz} (\wt M, \wt E)$. \\
Vassout's construction, which is carried out for positive operators, easily extends to any weight using an appropriate spectral cut.
\end{rk}

Families $z\mapsto Q^{-z}$  of complex powers are holomorphic, a notion we briefly recall.

\begin{defn}\label{defn:holfamilies} Let  $U$ be an open
subset of $\R^n$, let $V$ be a linear space and let $W$ be a domain in $\C$. A holomorphic
family of classical (also called polyhomogeneous) symbols  on $U$ with values in End$(V)$ parametrized by $W$ of order $\alpha:W\to \C$ is a function
$$\s(z)(x,\xi) := \s(z,x,\xi)\in\Ci(W \times U \times \R^n,
{\rm End} (V))$$ for which:
\begin{enumerate}
\item $\s(z)(x,\xi)$ is holomorphic at $z\in W$ as an element of
$\Ci(W \times U \times \R^n, {\rm End } V)$ and
\begin{equation}\label{e:logclassical}
 \s(z)(x,\xi) \sim \sum_{j\geq 0}
 \s_{\alpha(z)-j}(z)(x,\xi)
\end{equation}
is a classical symbol of order $\alpha(z)$ where the function  is holomorphic;
\item for any integer $N\geq 1$ the remainder
$\ds\sigma_{(N)}(z)(x, \xi):= \sigma(z)(x,\xi)- \sum_{j=0}^{N-1}
\sigma_{\alpha(z)-j}(z)(x, \xi)$ 
is holomorphic in $z\in W$ as an
element of $\Ci(W \times U \times \R^n, \End V)$ with $k^{\rm
th}$ $z$-derivative
\begin{equation}\label{e:kthderivlogclassical}
\sigma^{(k)}_{(N)}(z)(x, \xi) := \frac{\partial^k}{\partial z^k}(\sigma_{(N)}(z)(x, \xi)) 
\end{equation}
defining a locally uniform family of symbols of order $\Re(\alpha(z))-N+\ve$ in a compact neighborhood of any  $z_0 \in W$, for any positive $\e$. 
\end{enumerate}

A  family $z\mapsto A(z)$ in $\Psi_{\rm cl}(M,E)$ parametrised by a domain $W\subset \C$ is holomorphic
if in each local trivialisation of $E$ one has 
$$A(z) = {\rm Op}(\sigma(A(z)) + S(z)
$$ 
with $\sigma (A(z))$  a holomorphic
family of classical symbols and $S(z)$ a 
operator with Schwartz kernel $R(z,x,y)\in \Ci(W\times M\times
M,E\boxtimes E)$ holomorphic in $z$.
\end{defn}
There are at least two types of approaches  to show that complex powers define holomorphic families. One by Seeley (\cite{Se}, see also \cite[Thm. 11.2]{Shub}) 
%(CHECK that this is indeed a restriction, isn't it also axiomatic?), 
 using in a central manner the calculus of the symbol of the resolvent of the operator, and a cohomological  construction by Guillemin  \cite[Thm 5.2]{Gui} axiomatic in nature and  therefore easily transposable to more general contexts (see e.g. \cite{ALNV}). Note that Seeley and Shubin consider complex powers of differential elliptic operators but their construction can be extended to classical pseudodifferential operators using the symbol of the resolvent of these operators. Guillemin applies it to classical pseudodifferential operators with leading symbols that have a unique determination of the  logarithm.

 These constructions, which rely on basic properties of classical pseudodifferential operators, namely 

\begin{enumerate}
\item an estimate of the type \eqref{eq:estimateR}  on the resolvent of a weight,
leading to the existence of Cauchy integrals for weights, 
\item  the existence of a map $\Op:\sigma\mapsto \Op(\sigma)$ taking  symbols to operators in the algebra, that \lq\lq commutes" with  Cauchy integrals, 
\end{enumerate}   extend to very general algebras of classical pseudodifferential operators, including  classical pseudodifferential operators on groupoids and uniform classical pseudodifferential operators on coverings.

In particular, a weight   $\mathfrak Q\in \Upsi_{\Gamma, {\rm cl}}(\wt M, \wt E)$  gives rise to a holomorphic family
	\begin{equation}
	\label{eq:weightonlift}
 \mathfrak Q^{-z}=\left( \rho^\sharp\right)^{-1}\left( \rho^\sharp (\mathfrak Q)\right)^{-z},\end{equation}
where the map $\rho: \Upsi_{\Gamma, {\rm cl}}(\wt M, \wt E)\longrightarrow G(M, E)$ is defined in Appendix \ref{sec:groupoids}.

%%%%%%% %%%%%%% 
%%%%%%% %%%%%%% 
\section{Linear forms and trace-defect formulae on $\Gamma$-invariant operators }

\subsection{Local linear forms on classical pseudodifferential operators}
\label{section2}
Locality plays a fundamental role in the lifting procedure; in this section we single out \lq\lq local\rq\rq linear forms which can be lifted to coverings.

\subsubsection{The Wodzicki residue and canonical trace densities}
Let $X$ and $F$ be of bounded geometry as in the previous section, let $m\in\C$. To an operator $A={\rm Op}(\sigma(A))+ R(A)\in \Psi_{\rm cl}^m(X, F)$ and  $x\in M$ with   symbol $\sigma(A)(x,\cdot)\sim\sum_{j=0}^\infty \sigma_{m-j}(A) (x, \cdot)$ in a local trivialisation around $x$, where $\sigma_{m-j}(A),\; j\in \Z_{\geq 0}$  are the positively homogeneous components of the symbol of degree $m-j$,   we assign 
\begin{itemize}
\item the {\bf pointwise Wodzicki} residue $${\rm Res}_x(A):= \frac{1}{(2\pi)^n}\,\int_{\vert \xi\vert=1} {\rm tr}\left(\sigma_{-n}(A)(x,\xi)\right)\, d\xi,$$ 
\item the  {\bf pointwise canonical trace}  $${\rm TR}_x(A):= \frac{1}{(2\pi)^n}\, \cutoffint_{\R^n} {\rm tr}\left(\sigma(A)(x,\xi)\right)\, d\xi:=\frac{1}{(2\pi)^n}\,{\rm fp}_{R\to \infty} \int_{\vert \xi\vert\leq R} {\rm tr}\left(\sigma(A)(x,\xi)\right)\, d\xi,$$ where the abreviation fp  for finite part  means we pick  the constant term in the asymptotic expansion as $R$ tends to infinity and $\cutoffint_{\R^n}$ is the corresponding cut-off integral.  Here, tr stands for the fibrewise trace. \end{itemize}
We recall well-known facts.
\begin{lem}
\label{lem:Alocalised} {} \cite{W, KV} Given an operator $A\in \Psi_{\rm cl}(X,F)$, both  ${\rm Res}_x(A)\, dx$ and --whenever the order of $A$  does not lie in $[-n, +\infty[ \,\cap \,\Z$-- ${\rm TR}_x(A)\, dx$ define a global density on $X$.  
\end{lem}
\subsubsection{The Wodzicki residue and the canonical trace on closed manifolds}
From now on in this section we assume that  {\bf $X=M$ is a closed manifold} and $F=E$ is a vector bundle over $M$. We adopt the notations of \cite{Sc} and of Lemma \ref{lem:Alocalised}.

\begin{prop} 
Let  $A\in \Psi_{\rm cl}(M, E)$   have  local symbol  $\sigma_U(A)$ over any trivialising open subset $U$,  and let $(U_i, \chi_i)_{i\in I}$ be a partition of unity on $M$ subordinated to a trivialisation of $E$.
\begin{enumerate}
\item The Wodzicki residue density integrates to the Wodzicki residue \cite{W}
\begin{eqnarray}\label{eq:Resint}
{\rm Res}\left(A\right) 
%&:=&{\rm Res}\left(A_0\right)\nonumber \\
&:=& \int_{M} {\rm Res}_x( A) \, dx\\
&=&  \sum_{\supp \chi_i\cap\supp\chi_j\neq \emptyset}\int_{U_i\cap U_j}\chi_i(x)    {\rm Res}_x\left(\Op(\sigma_{U_i\cap U_j})(A)\right)\, \chi_j(x) dx. \nonumber
\end{eqnarray}
\item  Provided the order of the operator  does not lie in $\Z_n:=[-n,+ \infty[\,\cap\, \Z$, the canonical trace density integrates to the canonical trace \cite{KV,Le}
\begin{eqnarray}\label{eq:TRint}
{\rm TR}\left(A\right) 
%&:=&{\rm TR}\left(A_0\right)\nonumber \\
&:=&\int_{M}{\rm TR}_x(A)\, dx\\
&=&   \sum_{\supp \chi_i\cap\supp\chi_j\neq \emptyset}\int_{U_i\cap U_j}\chi_i(x) \, {\rm TR}_x
\left(\Op(\sigma_{U_i\cap U_j}(A))   \right)\,\chi_j(x)\, dx, \nonumber
\end{eqnarray} 
 \end{enumerate}
 which are both well-defined,  independently of the choice of  trivialisation and subordinated partition of unity.\\ When $A$ is trace-class, namely when the order of $A$ has   real part smaller  than $-n$, then ${\rm Res}(A)=0$ and $ {\rm TR}\left(A\right)= {\rm Tr}\left(A\right)$, where ${\rm Tr}$ stands for the ordinary trace of $A$.
Furthermore, we have
\begin{eqnarray*}
\label{eq:TRDelta}
A\sim B\Longrightarrow {\rm Res}\left(A\right)={\rm Res}\left(B\right)\quad &\text{and}&\quad A\underset{ \rm \small diag}{\sim} B\Longrightarrow {\rm TR}\left(A\right)={\rm TR}\left(B\right)
\\
 {\rm Res}(A)={\rm Res}(A_0)= {\rm Res}([A])\quad &\text{and}&\quad  {\rm TR}(A)={\rm TR}(A_0)= {\rm TR}([A]_{\rm diag}).
\end{eqnarray*}
\end{prop} 
\begin{proof}We refer  to \cite{W, KV} for the existence of Res and TR.  The  pointwise residue ${\rm Res}_x$   vanishes on operators  with smooth kernels and the pointwise canonical trace ${\rm TR}_x$ vanishes on operators with smooth kernel supported outside the diagonal.  Consequently, only $A_0:= \sum_{\supp \chi_i\cap\supp\chi_j\neq \emptyset}\chi_i A\chi_j$ arises in \eqref{eq:Resint} and \eqref{eq:TRint}.
% This same observation also implies  (\ref{eq:TRDelta}). 
\end{proof} 
\subsubsection{Characterisation of local linear forms on operators}
 As before, $ \pi:E\to M$ denotes a rank $k$-complex vector bundle over a closed $n$-dimensional manifold $M$.
\begin{defn} 
\label{def:sigmaclass}
We call a  {\bf $\Sigma$-class} ($\Sigma$-for symbol) any class $\Sigma(M, E)\in \Psi_{\rm cl}(M, E)$ of classical pseudodifferential operators    such that
\[A\in \Sigma(M, E)\, \wedge\, B\underset{\rm diag}{\sim}A \Longrightarrow B\in \Sigma(M, E),\] so that for an operator to  belong to $\Sigma$ is a condition on its symbol. Let $\Sigma_{\rm symb}\subset{\rm CS}(\R^n, {\rm gl}_k(\C)) $ be the corresponding class of symbols so that \[A\in \Sigma(M, E)\Longleftrightarrow A\sim_{\rm diag}  \Op(\sigma (A))\;\; \wedge \;\;\sigma(A)\in \Sigma_{\rm symb}. \]
\end{defn}
\begin{ex}
The algebra	$\Psi_{\rm cl}^{\Z}(M, E)$ of integer order classical operators and the set $\Psi_{\rm cl}^{\notin\Z}(M, E)$ of noninteger order classical operators form  $\Sigma$-classes.
\end{ex}
We define a class of linear forms (by linear form on a $\Sigma(M, E)$, we mean that $\Lambda(\alpha A+B)=\alpha \Lambda(A)+ \Lambda(B)$ for any scalar $\alpha$ and any operators $A$ and $B$ in $\Sigma$ such that $\alpha A+B$  also lies in $\Sigma$) on a  $\Sigma$-class $\Sigma(M, E)$, which only detect the symbol of the operator.  For this purpose, it is useful to introduce the corresponding class of  scalar valued symbols 
\[\Sigma_{\rm tr, symb}:=\{x\longmapsto {\rm tr}_x\circ \sigma(x,\cdot), \; \sigma\in \Sigma_{\rm symb}\}\subset {\rm CS}(\R^n, \C),\]
where ${\rm tr}_x$ is the fibrewise trace on the group   $\End(E_x)$ of endomorphisms of the fibre $E_x$ over the point $x$.
\begin{defn}\label{defn:locallambda}We  call {\bf local} any linear form $\Lambda\colon\Sigma(M, E)\to \C$ on a  $\Sigma$-class of classical pseudodifferential operators $\Sigma(M, E)$  which 
	\begin{itemize}
		\item is constant on $\underset{\rm diag}{\sim}$-equivalence classes:
		\[A\underset{\rm diag}{\sim} B\Longrightarrow \Lambda(A)=\Lambda(B),\] 
		so that \begin{equation}
		\label{eq:localLambda}
	\Lambda(A)=\Lambda( {\rm Op}(\sigma(A))),
		\end{equation}
		\item and such that
		\[\Lambda({\rm Op}(\sigma))=\int_M \lambda\circ {\rm tr}(\sigma(x,\cdot))\, dx\]  for some linear form $\lambda:\Sigma_{\rm tr, symb}\to \C$, so that
		
		\begin{equation}\label{eq:Lambda} \Lambda(A)=\int_M \Lambda_x(A)\, dx:=\int_M \lambda\circ {\rm tr}\left(\sigma(A)(x, \cdot)\right)\, dx.\end{equation}
	\end{itemize}
\end{defn}
\begin{rk} The adjective \lq\lq local\rq\rq{} reflects the fact that  the linear form $\Lambda$ is the integral of a density\[\omega_\Lambda(A)(x):=\lambda\circ {\rm tr}\left(\sigma(A)(x, \cdot)\right)\, dx\] over $M$.
	\end{rk}
\begin{ex}
The Wodzicki residue on integer order classical operators (resp. the canonical trace on non-integer order classical operators)  are local.
\end{ex}
 In the remaining part of this paragraph, $\Sigma(M, E)\subset \Psi_{\rm cl}(M,E)$  is a  $\Sigma$-class of operators which is {\bf closed} for the Fr\'echet topology of classical pseudodifferential operators of fixed order (see e.g. \cite[p.117]{Pa} and references therein).\begin{ex}  The  $\Sigma$-classes  $\Psi_{\rm cl}^\Z(M, E)$ of integer order classical pseudodifferential operators and  $\Psi_{\rm cl}^{\notin\Z}(M, E)$ of noninteger order classical pseudodifferential operators, both determined by conditions on their order, are closed for the Fr\'echet topology of classical pseudodifferential operators of fixed order.
 \end{ex}
 \begin{lem}Given a    $\Sigma$-class   $ \Sigma(M, E)\subset \Psi_{\rm cl}(M, E)$,
 	\begin{itemize}
 		\item the corresponding symbol class $\Sigma_{\rm symb}$ is invariant under scaling $\xi\longmapsto t\, \xi$ for $0<t<1$   and ${\rm O}_n(\R)$;
\item Given a   {\rm local and continuous} linear form  $\Lambda\colon\Sigma(M, E)\to \C$,  the associated linear form	$\lambda\colon\Sigma_{\rm tr, symb}\to \C$ on the corresponding class of scalar valued symbols  is continuous,  behaves covariantly under rescaling $\xi\longmapsto t\, \xi$ for any $0<t<1$  and ${\rm O}_n(\R)$-invariant. 
\end{itemize}
 \end{lem}
 \begin{proof}
    We first observe that  $\lambda$ is continuous for the Fr\'echet topology on symbols 
    of constant order as a result of the continuity of $\Lambda$.   Let us   deduce further properties of $\lambda$ from the covariance of $\Lambda$. \begin{enumerate}
 			\item  
 			Since $\omega_\Lambda(A)(x)$ defines a density,  for any local diffeomorphism $\kappa\colon U\to U$, $\kappa^*\omega_\Lambda(\kappa^\sharp A_U)= \omega_\Lambda( A_U)$, where  $A_U$ is a localisation of $A$ in that chart.
 			If $A_U$ has symbol $\sigma$, following the notations of \cite{Pa}, let  $\widetilde{\kappa_*\sigma}$ be the symbol of $\kappa^\sharp A_U$. The above covariance property for the form $\omega_\Lambda$ translates to \begin{equation}
 			\label{eq:kappa}\vert {\rm det}\kappa_* \vert\, \lambda\left(\widetilde {\kappa_*\sigma}\right)= \lambda\left( \sigma \right) 
 			\end{equation} for any  local diffeomorphism $\kappa:U\to U$.
 			\item Choosing $U$ to be an open ball and    $\kappa= t\,R$ for any $R\in O_n(\R)$ with   $0<t<1$,  it follows from the invariance property (\ref{eq:kappa}),  that  the symbol  class $\Sigma_{\rm symb} $ is invariant and the linear form $\lambda$ behaves covariantly under  rescaling $\xi\longmapsto t\, \xi$ for any $0<t<1$ as well as under isometric transformations.\qedhere
 			\end{enumerate}
 \end{proof}
In the following, continuity of local linear forms is defined with respect to the Fr\'echet topology of classical pseudodifferential operators of fixed order (see e.g. \cite[p.117]{Pa} and references therein.
\begin{thm}
\label{thm:uniqueness}
Any local continuous linear  form  on $\Psi_{\rm cl}^\Z(M, E)$   (resp.  $\Psi_{\rm cl}^{ \notin \Z}(M, E)$) is proportional to the Wodzicki residue ${\rm Res}$ (resp. the canonical trace ${\rm TR}$). 
\end{thm}
\begin{proof} The proof relies on results of \cite{Pa}, the basic ideas being that rescaling and isometric invariant linear forms   i) on homogeneous functions are proportional to the residue (see Lemma \cite[3.42]{Pa})  and ii) on Schwartz functions are proportional to the ordinary integral (see  Lemma \cite[3.40]{Pa}).  Let $k$ be the rank of $E$ as a   complex bundle over $M$. 
\begin{enumerate}
\item   We first  characterise $\Lambda $ on $\Psi_{\rm cl}^{\notin\Z}(M,E)$.
 Theorem 3.43 in \cite{Pa} characterises continuous rescaling and ${\rm O}_n(\R)$-invariant linear forms  on $CS^{\notin \Z}(\R^n)$ which turn out to be proportional to the canonical integral.  Note that the proof of  \cite[Lemma 3.42]{Pa} on which  in \cite[Theorem 3.43]{Pa} relies, only requires an invariance under rescaling by $0<t<1$.
So there is a constant $C$ such that $\lambda=C\, \cutoffint_{\R^n}$ on $CS^{\notin \Z}(\R^n)$. It follows that on $\Psi_{\rm cl}^{\notin \Z}(M,M\times\C)$, the linear form $\Lambda$ is proportional to the canonical trace TR. Composing with the trace on matrices yields the expected characterisation on $\Psi_{\rm cl}^{\notin\Z}(M,E)$.
\item Let us  now characterise $\Lambda $ on $\Psi_{\rm cl}^{\Z}(M,E)$.
The covariance assumption of Theorem 4.21  in \cite{Pa}-- which says that any  covariant continuous  linear form on $CS^\Z (\R^n)$ is proportional to the residue, the proof of which relies on Theorem 3.43-- can easily be relaxed to  an invariance under rescaling and isometric transformations. Thus,  any continuous linear form on $CS^\Z(\R^n)$ which is invariant under rescaling and isometric transformations is proportional to the residue ${\rm res}$. It follows that
$\Lambda$ on  $\Psi_{\rm cl}^\Z(M, M\times\C)$  is proportional to the Wodzicki residue ${\rm Res}$. Composing with the trace on matrices yields the expected characterisation on $
% \Cl^\Z(M,E)
\Psi_{\rm cl}^\Z(M, E)$. \qedhere
\end{enumerate}\end{proof}
We use this locality   in an essential way to lift trace defect formulae. Here is an immediate corollary which uses the known traciality  of  the Wodzicki residue and the canonical trace on the appropriate classes of operators they are defined on.
\begin{cor} Any local continuous linear form on    the class $\Psi^\Z(M, E)$ of integer order classical pseudodifferential operators, resp. on the class $\Psi^{\notin \Z}(M, E)$, is a trace, resp. vanishes on brackets.
\end{cor} 
\subsection{Trace defect formulae on closed manifolds  and applications}

\subsubsection{Trace defect formulae}
We recall (without proof) useful trace defect formulae for the canonical trace of holomorphic families of classical pseudodifferential operators \cite{KV,PS}.
\begin{prop}\label{thm:KVPS}  For any holomorphic family  $ A(z)\in \Psi_{\rm cl}(M,E)$ of classical operators parametrised by $\C$ with holomorphic order $-qz+a$ for some positive $q$ and some real number $a$,
   \begin{enumerate}
 \item the map    $z\mapsto {\rm TR} \left(A(z)\right)  $  is  meromorphic with simple poles $d_j:=\frac{a+n-j}{q}$, $ j\in \Z_{\geq 0}$,
 \item $\:$\cite{KV} the complex  residue at the point $d_j$   is given by  
 \begin{equation}\label{eq:classicalKV}{\rm Res}_{z=d_j} {\rm TR} \left(A(z)\right)= \frac{1}{q}{ \rm
 Res}  (A(d_j) ).
 \end{equation} 
 \item $\:$\cite{PS}  If $A(d_j)$ has a well-defined canonical trace ${\rm TR}\left(A  (d_j)\right)$, then $A^\prime (d_j)$ has a well defined 
Wodzicki residue  ${ \rm
  Res} (A^\prime(d_j) ) $     and  the constant term in the meromorphic expansion of ${\rm TR}\left(A  (z)\right)$ at $d_j$ is
  \begin{equation}\label{eq:PSclassicalop}
{\rm fp}_{z=d_j}{\rm TR}\left(A  (z)\right)={\rm TR}\left(A  (d_j)\right) + 
 \frac{1}{q}{ \rm
 Res} (A^\prime (d_j)).
 \end{equation} 
 In particular, this holds if $A(d_j)$ is a differential operator, in which case ${\rm TR}\left(A  (d_j)\right) =0$ and we have
 \begin{equation}\label{eq:PSclassicalopdiff}
{\rm fp}_{z=d_j}{\rm TR}\left(A  (z)\right)=
 \frac{1}{q}{ \rm
 Res} (A^\prime (d_j)).
 \end{equation} 
 \end{enumerate} 
\end{prop}  
\begin{rk}
	Actually, the derivatives $A^\prime(d_j)$ are log-polyhomogeneous operators (see \cite{Le} and references therein) and formula  (\ref{eq:PSclassicalopdiff})  yields an extension of the Wodzicki  residue to this particular operator, which differs from Lesch's graded residue.
\end{rk}
\subsubsection{The index as a trace defect}\label{subsec:index}
We specialise to  a Hermitian $\Z_2$-bundle $E=E_+\oplus E_-$ over a closed Riemannian manifold $M$ and let $D=:\left(\begin{array}{cc}0 & D^- \\D^+ & 0\end{array}\right)$ be an essentially selfadjoint elliptic differential operator 
% in $ \Psi_{\rm cl}(M, E)$  
of positive order $d$.  The operator $\Delta:=D_-D_+\oplus D_+D_-=D^2$ is essentially self-adjoint and has a finite dimensional kernel. % ${\rm Ker}(\Delta)$. 
 
Let $\Delta+S$ be an invertible perturbation of $\Delta $ by a smoothing operator $S\in \Psi_{cl}^{-\infty}(M,E)$; typically $S:=\chi_{\{0\}}(\Delta)$, the orthogonal projection onto the kernel of $\Delta$. The operator $\Delta+S$ is then a weight with principal angle $\theta=\pi$.
 Applying the $\mathbb Z_2$-graded version of Proposition \ref{thm:KVPS} to the family  $A(z)= \phi\,  (\Delta+S)^{-z}$ yields the following corollary. 
\begin{cor}\label{cor:KVPS} For any smooth function $\phi$ and any  invertible perturbation $\Delta+S$ as above, the Wodzicki residue
${ \rm
 sRes}(\phi\, \log (\Delta +S) )$ is well-defined and we have
\begin{equation}\label{eq:reslogdeltaR}{\rm fp}_{z=0}{\rm sTR}\left(\phi\, (\Delta+S)^{-z}\right)=- \frac{1}{2d}{ \rm
 sRes}_\theta(\phi\, \log( \Delta+S)).\end{equation}
\end{cor}
\begin{proof} This follows from Proposition \ref{thm:KVPS}  applied     to the family  $A(z)= \phi\,  (\Delta+S)^{-z}$ at $z=0$, with the fibrewise trace replaced by the $\Z_2$-graded fibrewise trace. 
\end{proof}
 Applying Corollary \ref{cor:KVPS} to the constant $1$-valued function $\phi$ yields the following formula for the index.
 
\begin{cor}\label{cor:indexD}\cite{Sc}
 \begin{equation}
\label{eq:indres}
{\rm ind} D^+=-\frac{1}{2d}{\rm sRes} ( \log ( \Delta +\chi_{\{0\}})).
\end{equation}
\end{cor}
\begin{proof}Expressing the index of $D_+$ as the supertrace of the projection operator onto the kernel we write
\begin{eqnarray}
{\rm ind} (  D_+)&=&{\rm sTr} (\chi_{\{0\}}(D))\nonumber\\
&=&{\rm sTR} \left((  \Delta+\chi_{\{0\}}(\Delta))^{-z}\right)\nonumber
\end{eqnarray}
since the nonzero eigenvalues of $ \; D_-D_+$ and $ D_+D_-$ coincide.  
This last expression defines a constant meromorphic function, whose finite part at zero therefore coincides with the index:
\begin{eqnarray}
{\rm ind} (  D_+)
&=& {\rm fp}_{z=0} {\rm sTr} \left((  \Delta+\chi_{\{0\}})^{-z}\right)  
\nonumber\\
&& \text{taking the finite part at zero} \nonumber \\
&=&-\frac{1}{2d}{\rm sRes} ( \log (  \Delta+\chi_{\{0\}})) \;\;\text{ using \eqref{eq:reslogdeltaR}}. \nonumber 
\end{eqnarray}
\end{proof}

%%% section
 \subsection{Trace defect formulae on coverings }\label{sec:Tdef-cov}
Let as before, $M$ be a (connected)   closed Riemannian manifold and $\wt M$ a  regular $\Gamma$-covering.
Let $\pi\colon E\to M$ be a Hermitian vector bundle and $\wt E:=\pi^* E\to \wt M$ its pullback.

 Let $\mathcal A \in \Psi_{ {\rm cl},\Gamma}(\wt M,\wt E)$ be a $\Gamma$-invariant classical pseudodifferential operator.
Let $x\in \wt M$, and  $\sigma(\mathcal A )(x,\cdot)$ denote the symbol in a local trivialisation around $ x$.  

Since $\mathcal A $ is $\Gamma$-invariant, the residue ${\rm Res}_x(\mathcal A )\, dx$ and  the canonical trace ${\rm TR}_x(\mathcal A )\, dx$  densities introduced  in Section \ref{section2} define  $\Gamma$-invariant densities, leading to the following definitions. 
 Let $F\subset X$ be a fundamental domain for the action of $\Gamma$ on $\wt M$.
\begin{defn}
\label{def:gammares}
 To  a $\Gamma$-invariant classical pseudodifferential operator   $\mathcal A \in \Psi_{ {\rm cl},\Gamma}^\Z(\wt M,\wt E)$, resp.  $\mathcal A \in \Psi_{ {\rm cl}, \Gamma}^{\notin \Z_n}(\wt M,\wt E)$  we assign the $\Gamma$-residue, resp. canonical $\Gamma$- trace
\begin{equation}
	\label{eq:covresiduedensity}
 {\rm Res}_\Gamma (\mathcal A ):= \int_{F}{\rm Res}_x (\mathcal A )dx, \quad {\rm resp.}\quad {\rm TR}_\Gamma (\mathcal A ):= \int_{F}{\rm TR}_x (\mathcal A )dx.
\end{equation}
 \end{defn}
  \begin{rk} 
  On operators in $\Psi_{ {\rm cl}, \Gamma} (\wt M,\wt E)$ of order smaller than $-n$,  which by \cite[Satz 4.4]{Va} are $\Gamma$-trace-class, the canonical $\Gamma$-trace coincides with the ordinary $L^2$-trace \begin{equation}\label{eq:Gammaordtrace}{\rm Tr}_\Gamma (\mathcal A ):= \int_{F}{\rm tr}\left(\sigma(\mathcal A )(x,\xi)\right)\, dx\, d\xi=\int_{F}{\rm tr}\left(K_{\mathcal A}(x,x)\right)\,dx,\end{equation}
  where $K_{\mathcal A}$ stands for the Schwartz kernel of ${\mathcal A}$.
  \end{rk} 
\begin{defnprop}
\label{prop:ResTRliftedA}
\begin{enumerate}
	\item A  $\Sigma$-class of operators $\Sigma(M, E)\subset \Psi_{\rm cl}(M,E)$ lifts to the class  of $\Gamma$-invariant operators
	\[\Sigma_\Gamma(\wt M, \wt E):=\{{\mathcal A}\in\Psi_{\Gamma,{\rm cl}}(\wt M,\wt E), \quad  {\mathcal A}\underset{\rm diag}{\sim} \wt A_0\quad \text{for some}\, A_0\in \Sigma(M, E)\},\]
	using the notation of (\ref{eq:abusenotation}).
\item Any local linear form $\Lambda\colon \Sigma(  M,   E) \to \C$ canonically lifts to a  linear form $\Lambda_\Gamma\colon \Sigma_\Gamma(\wt M, \wt E)\to \C$ defined as
\[\Lambda_\Gamma({\mathcal A}):= \Lambda(A)= \int_F \pi^*\left(\lambda\left({\rm tr}_x\sigma(A)(x,\cdot)\right)\, dx\right)\, \quad {\rm if}\,   
{\mathcal A}\in\reallywidetilde{ [A]_{\rm diag}}.\] 
\end{enumerate}
\end{defnprop}
\begin{rk}
	By construction local linear forms are constant on the horizontal lines of the diagramme in  \eqref{eq:diagramme}.
\end{rk}
\begin{proof} On an operator $ \wt A_0$  for some $\e$-local operator $A_0\in \Sigma(M, E)$, the $\Gamma$-invariant linear form is defined as $ \Lambda_\Gamma({\mathcal A}):=\Lambda(A_0)$.
	This determines $\Lambda_\Gamma$ on $\Sigma_\Gamma(\wt M, \wt E)$ since    a local linear form   is  constant on $\sim_{\rm diag}$-equivalence classes and we can set
	\[	\Lambda_{\Gamma}(\wt{[A]_{\rm diag}})=\Lambda([A]_{\rm diag}):=\Lambda(A_0)=\Lambda(A)\, ,\quad\text{for any $\e$-local operator}\, A_0\in [A]_{\rm diag}.\qedhere\]
\end{proof}
This applies the canonical trace and the Wodzicki residue.
\begin{cor} 
 The  residue and the canonical trace densities are well-defined on lifted classes $\reallywidetilde {[A]_{\rm diag}}$ of Definition \ref{defn:classlift} and are preserved under lifts.
  \[{\mathcal A}\in  \reallywidetilde{[ A]_{\rm diag}}\Longrightarrow  {\rm Res}_\Gamma({\mathcal A})= {\rm Res}(A)\quad A\in \Psi_{\rm cl}^\Z(M,E),\]
and \[{\mathcal A}\in  \reallywidetilde{[ A]_{\rm diag}}\Longrightarrow  {\rm TR}_\Gamma({\mathcal A})= {\rm TR}(A)\quad A\in \Psi_{\rm cl}^{\notin \Z}(M,E).\]  
 In other words, the residue and the canonical trace are constant along the horizontal lines of the diagramme (\ref{eq:diagramme}) in the introduction. 
\end{cor}

Consequently,   trace-defect formulae on closed manifolds recalled in Proposition \ref{thm:KVPS}  induce trace-defect formula on coverings.

\begin{thm}
\label{thm:KVPScov}
   For any holomorphic family  $ \mathcal A(z)\in \Psi_{{\rm cl},\Gamma}(\wt M, \wt E)$ of classical operators parametrised by $   \C$ with holomorphic order $   -qz+a$ for some positive $q$ and some real number $a$,
   \begin{enumerate}
 \item the map    $z\mapsto {\rm TR}_\Gamma \left(\mathcal A (z)\right)  $  is  meromorphic with simple poles $d_k:=\frac{a+n-k}{q}$, $ k\in \Z_{\geq 0}$,
 \item  the complex  residue at the point $d_k$   is given by
 \begin{equation}\label{eq:classicalKVnc}{\rm Res}_{z=d_k} {\rm TR}_\Gamma \left(\mathcal A(z)\right)= \frac{1}{q}{ \rm
 Res}_\Gamma (\mathcal A (d_k) );
 \end{equation} 
 \item   If $\mathcal A(d_k)$ has a well-defined $\Gamma$-canonical trace ${\rm TR}_\Gamma(\mathcal A(d_k))$, then $\mathcal A^\prime (d_k)$ has a well defined $\Gamma$-
Wodzicki residue  ${ \rm
  Res}_\Gamma (\mathcal A^\prime(d_k) ) $     and  the constant term in the meromorphic expansion of ${\rm TR}_\Gamma\left(\mathcal A  (z)\right)$ at $d_k$ is
  \begin{equation}\label{eq:PSclassicalop-bis}
{\rm fp}_{z=d_k}{\rm TR}_\Gamma\left(\mathcal A  (z)\right)=  {\rm TR}_\Gamma(\mathcal A(d_k))+
 \frac{1}{q}{ \rm
 Res}_\Gamma (\mathcal A^\prime (d_j));
 \end{equation} 
\item If $\mathcal A(d_k)$ is a differential operator, this reduces to 
 \begin{equation}\label{eq:PSclassicalopdiff-bis}
{\rm fp}_{z=d_k}{\rm TR}_\Gamma\left(\mathcal A  (z)\right)=  
 \frac{1}{q}{ \rm
 Res}_\Gamma (\mathcal A^\prime (d_j));
 \end{equation} 
 \item All the above statements actually hold for any other representative in  $[\mathcal A(z)]_{\rm diag}$.
 \end{enumerate} 
\end{thm} 
\begin{proof} We use Proposition \ref{prop:ShubinTh1} to write $\mathcal A (z)\underset{ \rm \small diag}{\sim}\reallywidetilde {A(z)}$ with $A(z)$ $\ve$-local for some positive $\ve$. On the grounds of Proposition \ref{prop:ResTRliftedA}, without loss of generality  we can show the statements for $\mathcal A(z)=\reallywidetilde {A(z)}$, and we have
\begin{equation}
\label{eq:TRGamma}
 {\rm TR}_\Gamma\left( \widetilde{A(z)}\right)={\rm TR} \left(  A(z) \right). 
 \end{equation}
 (1) then follows from the meromorphicity and the structure of the poles  $\{d_k, k\in \Z_{\geq 0}\}$ of ${\rm TR}  (A(z))$ discussed in Proposition \ref{thm:KVPS}.
Similarly, (\ref{eq:classicalKVnc})  follows from (\ref{eq:classicalKV}) combined with (see Proposition \ref{thm:KVPS})  
\begin{equation}
\label{eq:ResGamma}
{\rm Res}_\Gamma\left( \widetilde{A(d_j)}\right)={\rm Res} \left(  A(d_j) \right). 
\end{equation}
To prove (3) we assume that  $A(d_k)$ has  a well-defined canonical  trace, in which case  $\widetilde{A(d_k)}$ has a well-defined $\Gamma$-canonical trace. We apply  (\ref{eq:PSclassicalop-bis}) to  the family
  $ A(z)$, which  combined with (\ref{eq:TRGamma}) yields
\begin{eqnarray*}
 {\rm fp}_{z=d_k} {\rm TR}_\Gamma(\widetilde{A(z)})&=& {\rm fp}_{z=d_k}{\rm TR} \left(  A (z) \right)\\
 &=&{\rm TR} \left(  A(d_k) \right)+\frac{1}{q}   {\rm Res} \left(  A^\prime(d_k) \right)\\
 &=&   {\rm TR} \left( \reallywidetilde{ A(d_k)} \right)+\frac{1}{q}   {\rm Res} \left(  A^\prime(d_k) \right).
 \end{eqnarray*} 
 Since the operator $  A^\prime(d_k)$ has a well-defined residue, the lifted derivative $\reallywidetilde{A^\prime(d_k)} $ has a well-defined $\Gamma$-residue and we have    
$\ds	  {\rm Res}_\Gamma\left(\reallywidetilde{A^\prime(d_k)}\right)={\rm Res} \left(  A^\prime (d_k) \right),$ leading to (\ref{eq:PSclassicalopdiff-bis}). \\
 The fact that these statements depend only on the class $[\mathcal A(z)]_{\rm diag}$ and not on the representative follows from the fact that $\TR$ and ${\rm Res}$ are well defined on such classes.  \qedhere
\end{proof}
\begin{cor}\label{cor:KVPScomparison} Let  $A(z)$ be a holomorphic family of   operators in $\Psi_{\rm cl}(M, E)$. For some  positive number $\ve$,  there is a holomorphic family $A_0(z) $ of $\ve$-local operators  in $\Psi_{\rm cl}(M, E)$ such that $A_0(z)\in [A(z) ]_{\rm diag}$.\\
For any  holomorphic family  $\mathcal  A(z)$  in $\Psi_{{\rm cl},\Gamma}(\wt M, \wt E)$ such that the difference   $\mathcal  A (0)-\reallywidetilde{A_0(0)}$ at zero has a smooth kernel,  the map $z\mapsto {\rm TR}_\Gamma\left(\mathcal  A(z)\right)- {\rm TR} \left(A(z)\right)$ is holomorphic at $0$
with
\begin{equation}\label{eq:RestildeAz}
{\rm fp}_{z=0}{\rm TR}_\Gamma\left(\mathcal  A (z)\right)-{\rm fp}_{z=0}{\rm TR}\left(A(z)\right)= {\rm Tr}_\Gamma(\mathcal  A(0)-\reallywidetilde{A_0(0)}),
%+\frac{1}{q}\,{\rm Res}_\Gamma\left((\widetilde A)^\prime(p)-\wt{A^\prime(p)}\right)
\end{equation}
where  as before, ${\rm Tr}_\Gamma $ is defined in \eqref{eq:Gammaordtrace}.
 \end{cor} 
\begin{proof}  
 For some given positive $\ve$,  Definition \ref{defn:classlift} yields a family  of $\ve$-local operators  $A_0(z)$ in $ [A(z)]_{\rm diag}$. The explicit construction of these operators by means of an appropriate partition of unity shows that this family can be chosen holomorphic as a consequence of the holomorphicity of $A(z)$.  It then follows from Proposition \ref{prop:ResTRliftedA} that
	$${\rm TR}_\Gamma\left(\reallywidetilde{A_0(z)}\right)= {\rm TR} \left( A_0(z)\right)={\rm TR} \left( A(z)\right).$$
	   The operators $\mathcal B(z):=  \mathcal  A(z)- \reallywidetilde{A_0(z)}$   define a holomorphic family in  $\Psi_{{\rm cl},\Gamma}(\wt M, \wt E)$. It follows from the above that
	   \begin{equation}
	   	\label{eq:1}	{\rm fp}_{z=0}{\rm TR}_\Gamma\left(\mathcal  A (z)\right)-{\rm fp}_{z=0}{\rm TR}\left(A(z)\right)=	{\rm fp}_{z=0}{\rm TR}_\Gamma\left(\mathcal  A (z)\right)-{\rm fp}_{z=0}{\rm TR}\left(A_0(z)\right)={\rm fp}_{z=0}{\rm TR}_\Gamma\left(\mathcal  B (z)\right).
	   	\end{equation}  Assuming that  the operator  $ \mathcal  A(0)- \reallywidetilde{A_0(0)}$ has a smooth kernel, then the operator $\mathcal B(0) $ has a smooth kernel and hence a vanishing Wodzicki residue. It then follows from  (\ref{eq:classicalKVnc}) that the map $z\mapsto {\rm TR}_\Gamma\left(\mathcal B(z)\right)$  has a vanishing complex residue at zero, showing its holomorphicity at zero. 
The fact that $\mathcal B(0)$ has a smooth kernel also implies that it has  a well-defined canonical trace  which coincides with the ordinary $\Gamma$-trace ${\rm Tr}_\Gamma\left(\mathcal B(0)\right)$  and we have
 \begin{equation}
 \label{eq:2}{\rm fp}_{z=0}{\rm TR}_\Gamma\left(\mathcal  B (z)\right)=\lim_{z\to 0}{\rm TR}_\Gamma\left(\mathcal  B (z)\right)={\rm Tr}_\Gamma(\mathcal B(0))= {\rm Tr}_\Gamma(\mathcal  A(0)- \reallywidetilde{A_0(0)}).\end{equation} Combining equations (\ref{eq:1}) and (\ref{eq:2}) yields (\ref{eq:RestildeAz}).
\end{proof}
%%%%%%% %%%%%%% 
%%%%%%% %%%%%%% 
%%%%  section

%%%%  section

\subsection{Lifted spectral $\zeta$-invariants} 
\label{sec:liftedregtr}
This section is devoted to applications of Formula (\ref{eq:RestildeAz}) in Corollary \ref{cor:KVPScomparison}.\\
 Let $E$ be a hermitian vector bundle over a closed Riemannian manifold $M$, and let  $\mathbf D$ be  an essentially self-adjoint  differential operator in $\Psi(M, E)$. Then $\mathbf \Delta:= \mathbf D^2$ is a non-negative essentially self-adjoint operator.\\ 
 For any $\ve>0$, let $Q_\ve(\mathbf \Delta)$ be a weight as defined in \eqref{eq:Qeps} and for a  measurable  function $h$ on  on a contour $\Gamma_\pi$ around the ray $R_\pi=]-\infty, 0]$,  such that 
\begin{equation}
\label{eq:estimate.h}
 \vert h(\lambda)\vert \leq\vert  \lambda\vert^{-\delta} 
\end{equation}
for some positive $\delta$,  let $h(Q_\ve(\mathbf \Delta))$ be defined by \eqref{eq:hQ}. 
  Similarly, we consider $ h(Q_\ve(\wt{\mathbf \Delta}))$ with $Q_\ve(\wt {\mathbf \Delta})$  defined by (\ref{eq:tildeQeps}). \\
In the specific case when $h$ is polynomial, then $h(\mathbf \Delta)$  is a well-defined differential operator, which lifts to $\reallywidetilde{h(\mathbf \Delta)}=h(\wt{ \mathbf \Delta})$.
 
Let $h$  be a measurable  function $h$ on   a contour $\Gamma_\pi$ around the ray $R_\pi=]-\infty, 0]$,  satisfying \eqref{eq:estimate.h}  for some positive $\mathbf \Delta$, we set  $h_\ve(\mathbf \Delta):= h(Q_\ve(\mathbf \Delta))$ and $h_\ve(\wt{ \mathbf \Delta}):= h(Q_\ve(\wt{ \mathbf \Delta}))$. 
\\
	For any $A\in \Psi_{\rm cl}(M, E)$  and any weight $Q\in \Psi_{\rm cl}(M, E)$ , the family $A(z):= A\, Q^{-z}$ is a holomorphic  perturbation of $A$ and the {\bf spectral $\zeta$-function}  (or ${  Q}$-regularised trace of ${  A}$)
	\begin{equation}\label{eq:zetaM} z\longmapsto\zeta_{A,Q}(z):={\rm TR}(A\, Q^{-z})
	\end{equation} defines a meromorphic map  on $\C$  with a known countable set of simple poles.  
	Similarly, for any ${\mathcal A}\in \Psi_\Gamma(\widetilde M, \widetilde E)$ and any weight $ {\mathfrak Q}\in \Psi_\Gamma(\widetilde M, \widetilde E)$,  using (\ref{eq:weightonlift}) we define the holomorphic perturbation
	\[{\mathcal A}(z):= {\mathcal A}\, {\mathfrak Q}^{-z}\]  of ${\mathcal A}$ and thanks to  Theorem  \ref{thm:KVPScov}, we know that 
		the  spectral $\zeta$-function (or ${\mathfrak Q}$- regularised trace of ${\mathcal A}$) \begin{equation}\label{zetatildeM}z\longmapsto\zeta^\Gamma_{{\mathcal A},{\mathfrak Q}}(z):={\rm TR}_\Gamma({\mathcal A}\, {\mathfrak Q}^{-z})
		\end{equation}
		is meromorphic with a known countable set of simple poles.
	So we can take the  finite parts of the Laurent expansions at zero and build (with some abuse of notation)   {\bf spectral $\zeta$-invariants} 
		\begin{equation}\label{zetazero}
	\zeta_{A,Q}(0):={\rm fp}_{z=0}\left(	\zeta_{A,Q}(z)\right);\quad 	\zeta^\Gamma_{ {\mathcal A},{\mathfrak Q}}(0):={\rm fp}_{z=0}\left(\zeta^\Gamma_{{\mathcal A},{\mathfrak Q}}(z)\right).
		\end{equation}
		 Take for $\ve>0$ \[A_\ve ( \mathbf  D ):= P(\mathbf  D)\, h_\ve(\mathbf \Delta);\,{\mathcal A}_\ve(\widetilde{\mathbf  D})= P(\widetilde{\mathbf  D})\, h_\ve(\widetilde{\mathbf  \Delta}) .\]  \begin{rk}Recall from Remark \ref{rk:simDelta} that whereas the  equivalence relation $\sim$ is stable under products of operators, the equivalence relation $\underset{ \rm \small diag}{\sim}$ is not.
		 \end{rk}
In particular, we do not a priori expect  the regularised trace   $\zeta_{A_\ve ( \mathbf  D ), Q_\ve(\mathbf \Delta)}(0)$ to
  coincide  with the regularised  $\Gamma$-trace  $	\zeta_{{\mathcal A}_\ve(\widetilde{\mathbf  D}),{\mathfrak Q}_\e(\widetilde{\mathbf  \Delta})}(0)$, which involves the product $P(\wt D)\,h_\ve(\widetilde {\mathbf \Delta}) \,Q_\ve({\wt {\mathbf  \Delta}} )^{-z}$.\\
The following theorem nevertheless relates the two regularised traces.

\begin{thm}\label{thm:liftedregtraces} Let $\ve >0$. With the above notations 
\begin{itemize} 
\item The meromorphic map
$$z\longmapsto \zeta_{A ( \mathbf  D ), Q_\ve(\mathbf \Delta)}(z)-	\zeta^\Gamma_{{\mathcal A}(\widetilde{\mathbf  D}),{\mathfrak Q}_\e(\widetilde{\mathbf  \Delta})}(z)$$
is   holomorphic at zero.
\item There is  some positive $\alpha$, and  an $\alpha$-local operator $A_{\ve,0} ({\mathbf  D})\in\left[A_\ve({\mathbf  D}) \right]_{\rm diag}$ such that the difference    $ \reallywidetilde{A_{\ve,0}({\mathbf  D})}-{\mathcal A}_\ve(\wt{\mathbf  D})$ has a smooth kernel and hence a well-defined trace ${\rm Tr}_\Gamma\left(\reallywidetilde{A_{\ve,0}(\mathbf D)}-{\mathcal A}_\ve(\wt{\mathbf  D}))\right)$. 
We have
\begin{eqnarray}\label{eq:liftedregtraces}
  \zeta_{A_\ve ( \mathbf  D ), Q_\ve(\mathbf \Delta)}(0)-	\zeta^\Gamma_{{\mathcal A}_\ve(\widetilde{\mathbf  D}),{\mathfrak Q}_\ve(\widetilde{\mathbf  \Delta})}(0) = {\rm Tr}_\Gamma(\reallywidetilde{A_{\ve,0}({\mathbf  D})}-{\mathcal A}_\ve(\wt{\mathbf  D})),
\end{eqnarray}
\item Spectral $\zeta$-invariants   canonically lift to the covering. In other words, when  $h\equiv 1$, identity (\ref{eq:liftedregtraces}) amounts to the coincidence of the zeta-regularised trace and its lifted counterpart:
    \begin{equation}
   \label{eq:diffliftedregtraces} \zeta_{P ( \mathbf  D ), Q_\ve(\mathbf \Delta)}(0)=	\zeta^\Gamma_{P(\widetilde{\mathbf  D}),{\mathfrak Q}_\e(\widetilde{\mathbf  \Delta})}(0) . \end{equation} 
   \item The above statements extend to super-regularised traces for operators  acting on sections of a $\Z_2$-graded bundle, replacing  the canonical trace ${\rm TR}$ by a $\Z_2$-graded canonical tace ${\rm sTR} $ and correspondingly  $\zeta-$ functions by $\Z_2$-graded $\zeta-$ functions $s\zeta_{A,Q}(z):={\rm sTR}(A\, Q^{-z})$.
   \end{itemize}
\end{thm}

\begin{proof}The operators $ Q_\ve (\mathbf \Delta)^{-z}$, resp.  $ Q_\ve (\wt { \mathbf \Delta})^{-z}$ built from complex powers define holomorphic families on $M$, resp. $\wt M$ and hence so do the  operators $  B_\ve  (\mathbf D)(z):=  A_\ve  (\mathbf D)\,Q_\ve (\mathbf \Delta)^{-z} $, resp. ${\mathcal B}_\ve(\wt{\mathbf D})(z):= {\mathcal A}_\ve (\wt{\mathbf D})\,Q_\ve (\wt { \mathbf \Delta})^{-z}$  define   homolorphic families on $M$, resp. $\wt M$, to which we want to apply Corollary \ref{cor:KVPScomparison}. To simplify notations, we drop the index $\ve$ and the operator ${\mathbf D}$ from the notation, setting   $   B(z):= B_\ve(z) (\mathbf D)$ and
	${\mathcal B}(z):={\mathcal B}_\ve(z)(\wt{\mathbf D})$. With these notations
\begin{equation}\label{eq:useful1}	\zeta_{A_\ve ( \mathbf  D ), Q_\ve(\mathbf \Delta)}(0)-	\zeta_{{\mathcal A}_\ve(\widetilde{\mathbf  D}),{\mathfrak Q}_\e(\widetilde{\mathbf  \Delta})}(0)= {\rm TR}_\Gamma\left(\mathcal B(z)\right)- {\rm TR} \left(  B(z)\right).\end{equation}
On the one hand and  as in Corollary \ref{cor:KVPScomparison}, we build a $\ve$-local holomorphic family $B_0(z)\in [B(z)]_{\rm diag}$. In particular,  the operator $B_0:=B_0(0)= {A}_\ve( {\mathbf D})$ has symbol $ P(\sigma({\mathbf  D}))\star h_\star (\sigma(Q_\ve( \mathbf \Delta))$. So, on the one hand its lift $\wt{B_0}$ has symbol $\reallywidetilde{P(\sigma({\mathbf  D}))\star h_\star (\sigma(Q_\ve( \mathbf \Delta))}$. 
On the other hand, the operator $\mathcal{B}(0)=\mathcal{A}_\ve(\wt{\mathbf D})$ has  symbol $P(\sigma(\wt{\mathbf  D}))\star h_\star(\sigma(Q_\ve(\wt{ \mathbf \Delta})))$.  By  (\ref{eq:hstarQ}) applied to $\mathfrak Q=Q_\ve(\wt{\mathbf \Delta})$ and $Q=Q_\ve(\mathbf \Delta)$, we have $h_\star(\sigma(Q_\ve(\wt{ \mathbf \Delta})))\sim \reallywidetilde{h_\star (\sigma(Q_\ve( \mathbf \Delta))}$.  
 Consequently, using again the locality of $\star$ as in Lemma \ref{lem:hstarQ}, we find that the operator  $\mathcal{B}(0)$ has the same symbol as $\reallywidetilde{B_0}$. Hence the difference   $\mathcal{B}(0)-\reallywidetilde{B_0}$ lies in $\cap_{m\in \R}\Upsi_\Gamma^{m}(\wt M, \wt E)$ as a result of which it is  $\Gamma$-trace-class and therefore has a well-defined $\Gamma$-trace ${\rm Tr}_\Gamma \left(\mathcal B(0)-\wt{B_0}\right)$ (see \ref{eq:Gammaordtrace}).\\ Applying Corollary \ref{cor:KVPScomparison}   yields the holomorphicity at $z=0$ of the map in (\ref{eq:useful1})
 $$ 	\zeta_{A_\ve ( \mathbf  D ), Q_\ve(\mathbf \Delta)}(0)-  	\zeta_{{\mathcal A}_\ve^\Gamma(\widetilde{\mathbf  D}),{\mathfrak Q}_\e(\widetilde{\mathbf  \Delta})}(0)= {\rm fp}_{z=0}\left( {\rm TR}_\Gamma(\mathcal B(z))- {\rm  TR}_\Gamma(\reallywidetilde{B_0(z)})\right)
={\rm Tr}_\Gamma \left(\mathcal B(0)-\wt{B_0}\right),$$  which corresponds to (\ref{eq:liftedregtraces}).\\
If $h\equiv 1$ , then $B_0=B(0)= P({\mathbf  D})$ is a differential operator and we have ${\mathcal  B}(0) =P(\wt{\mathbf  D}) = \reallywidetilde{P({\mathbf  D})} =\reallywidetilde{B(0)}=\reallywidetilde{B_0} $ so that ${\rm Tr}_\Gamma({\mathcal  B}(0) -\wt{B_0})=0$, from which the assertion \eqref{eq:diffliftedregtraces} follows.
	\item Replacing ${\rm TR}$ by its $\Z_2$-graded analog ${\rm sTR}$ yields the last assertion. \qedhere.
\end{proof} 

\subsubsection{ The $L^2$-Atiyah index theorem revisited}\label{subsec:L2index}
 We apply the above construction to  a hermitian $\Z_2$-bundle $E=E_+\oplus E_-$ over a closed Riemannian manifold $M$, so that its pull-back $F:=\wt E=\wt E_+\oplus \wt E_-$  by $\pi$ is a  $\Z_2$ graded $\Gamma$-equivariant vector bundle over $X:= \wt M$. \\
 Let $ D_+: C^\infty( M, E_+)\to C^\infty( M,  E_-)$ be an elliptic differential operator of order $d$ with   formal adjoint $ D_-:  C^\infty( M, E_-)\to C^\infty( M, E_+)$. Let  $\wt D_\pm: C^\infty(X,F_\pm)\to C^\infty(X, F_\mp)$ be the lifted differential operators.  The operator   \begin{equation}\label{eq:wtD}\wt D:
	= \begin{bmatrix}
	0 & \wt D_-  \\
	\wt D_+ & 0  
	\end{bmatrix} 
\end{equation} is a $\Gamma$-invariant  elliptic differential operator of positive order $d$ to which we apply the above constructions.

  Even though the kernels $\{s\in C_c^\infty(\wt M, \wt E), \; \wt D_\pm s=0\}$  are not necessarily   finite dimensional, their closures $K_{\wt D_\pm}$   are   finitely generated $\Gamma$-modules  and  hence isometrically
$\Gamma$-isomorphic to Hilbert
$\Gamma$-subspaces of the Hilbert
space $\ell_2(\Gamma)^n$ for some positive integer $n$, which  can be represented by   idempotent matrices $P^\pm=(p_{ij}^\pm)\in {\rm gl}_n\left({\mathcal N}\Gamma\right)$. Let  $\chi_{\{0\}}(\wt D_\pm)=\chi_{\{0\}}(\wt \Delta)$ denote the orthogonal projections  onto $K_{\wt D_+}\oplus K_{\wt D_-}$.

The $\Gamma$-dimension (resp. $\Gamma$- graded dimension) of $K_{\wt D_\pm}$  is (see Appendix \ref{sec:appHM})
 $$
 {\rm dim}_\Gamma K_{\wt D_\pm}:= \sum_{i=1}^n\langle p_{ii}^\pm(e), e\rangle, \quad {\rm resp.} \quad {\rm sdim}_\Gamma (K_{\wt \Delta}):=  {\rm dim}_\Gamma K_{\wt D_+}-  {\rm dim}_\Gamma K_{\wt D_-}={\rm sTr}_\Gamma\left(\chi_{\{0\}}( \wt \Delta)\right),
 $$
where $e\in \C\Gamma$ is the element with all components zero outside the first one which is one and  the $\Gamma$-index of $D$ is
$${\rm ind}_\Gamma(\wt D):= {\rm dim}_\Gamma K_{\wt D_+}-{\rm dim}_\Gamma K_{\wt D_-}.
$$
The $\Gamma$-Wodzicki residue and the $\Gamma$-canonical trace extend in a straightforward manner to a $\Z_2$-graded Wodzicki residue ${\rm  sRes}_\Gamma$ and a $\Z_2 $-graded canonical trace ${\rm sTR}_\Gamma$.

\begin{cor}
\label{cor:indres2}With the notation of \eqref{eq:wtD}, $   \log( Q_\ve(\wt\Delta)  )$ has a well-defined (super) $\Gamma$-residue and  we have
 \begin{equation}
 \label{eq:indres2}{\rm ind}_\Gamma(\wt D_+)= -\frac{1}{q}{ \rm
 sRes}_\Gamma  (  \log(  Q_\ve(\wt\Delta)  ))=-\frac{1}{q}\,{ \rm sRes}(  \log(  Q_\ve(\Delta)  ))={\rm ind}(D_+).
\end{equation}
\end{cor}
\begin{proof}   
By Corollary \ref{cor:indexD}, the index is a $\Z_2$-graded regularised trace of the identity:
${\rm ind} (D_+)= {\rm sTR}^{Q_\ve} (Id)$ and similarly, independently of $ \ve >0$, we have
\begin{eqnarray}
{\rm ind}_\Gamma(\widetilde D_+)&=&
{\rm sTR}_\Gamma \left(\chi_{\{0\}}(\wt \Delta)\right) \nonumber\\
&=&{\rm sTR}_\Gamma \left(\chi_{[-\ve, \ve]}(\wt \Delta)\right) \nonumber={\rm sTR}_\Gamma (  Q_\ve(\wt\Delta)^{-z} ) \nonumber
% \\
%&&
\;\;\text{seen as meromorphic functions}\\
&=&{\rm fp}_{z=0} {\rm sTR}_\Gamma\left(\left(Q_\ve (\wt \Delta)\right)^{-z}\right) \nonumber \;\text{ taking the finite part at zero}\nonumber\\
&=&  {\rm sTR}^{Q_\ve(\wt \Delta)}(\wt Id).
\end{eqnarray} 
 Theorem \ref{thm:liftedregtraces} applied to $h\equiv 1$  and ${\mathbf  D}\equiv 1$ then yields ${\rm ind}_\Gamma(\wt D_+)={\rm ind}_\Gamma(D_+)$. The   lifted  trace-defect formula (\ref{eq:PSclassicalop-bis}) derived in Theorem \ref{thm:KVPScov} applied to the family ${\mathcal A}(z):=\left(Q_\ve (\wt \Delta)\right)^{-z}$  further  yields the expression of ${\rm ind}_\Gamma(\widetilde D_+)$ as a Wodzicki residue.
\end{proof}
\subsubsection{The $\eta$-invariant revisited}
We now assume that both $D$ and $\wt D$ are essentially self-adjoint and  invertible, in which case   $Q:=D^2$ is a weight which lifts to  $\wt Q= \wt D^2$.  The operators $\vert D\vert^{-1}=\Delta^{-\frac{1}{2}}$ and $\vert\wt D\vert^{-1}=\wt\Delta^{-\frac{1}{2}}$  are defined as  Cauchy integrals (see (\ref{eq:hQ})) using  $h(x)= x^{-\frac{1}{2}}$ and the  $\eta$-invariant of $D$ can be expressed in terms of regularised traces \cite{CDP} as
$\ds\eta(D)= {\rm Tr}^Q(D\,\vert D\vert^{-1}); \quad \eta_\Gamma(\wt D)= {\rm Tr}_\Gamma^{\wt Q}(\wt D\,\vert\wt D\vert^{-1}).$ 
\begin{cor}
\label{cor:eta} There is an $\ve$-local operator $A_0\in 
\left[D\, \vert D\vert^{-1}\right]_{\rm diag}$ for some  small enough positive $\ve$ (see Definition \ref{defn:classlift}), such that the difference $\reallywidetilde{A_0(D)}-\wt D\, \vert \wt D\vert^{-1}$ has a smooth kernel and a well-defined $\Gamma$-trace and we have 
 
$$\eta(D)-\eta_\Gamma(\wt D)=  {\rm TR}_\Gamma\left(\reallywidetilde{A_0(D)}-\wt D\, \vert \wt D\vert^{-1}\right).$$
 \end{cor}
 
\begin{proof} The statement is a straightforward consequence of Theorem \ref{thm:liftedregtraces} applied to $P(x)=x$ and $h(x)=x^{-\frac{1}{2}}$, and the right hand side does not depend on the representative $A_0$ in $ [D\, \vert D\vert^{-1}]_{\rm diag}$.
\end{proof}
\subsection{Invariants built from geometric operators}
Let $F \to X$ be a vector bundle over a  Riemannian manifold $\left(X,g\right)$. We assume that $X$ is spin and $F=S\otimes W$ where $S$ is the spinor bundle and $W$ an auxillary bundle equipped with a connection $\nabla^W$. This way, $F$ can be equipped with a connection $\nabla^F:= \nabla\otimes 1+ 1\otimes \nabla^W$, where $\nabla$ is the Levi-Civita connection on $S$. \\ Following Gilkey's notations \cite[Formula (2.4.3)]{G}, for a multi-index $\alpha=(\alpha_1,\cdots, \alpha_s)$
we introduce formal variables $g_{ij/\alpha}:= \partial_\alpha g_{ij}$ 
for the partial derivatives of the metric
tensor $g$ on the manifold $X$
and  similarly $\omega_{i/\alpha}:= \partial_\alpha \omega_{i}$ for  the connection
$\omega$
on the external bundle. Let us set
${\rm ord}\left(g_{ij/\alpha}\right)= \vert \alpha\vert=\alpha_1+\cdots+\alpha_s 
; {\rm  ord}(\omega_{i/\beta})=\vert \beta\vert$. 
Inspired by Gilkey \cite[Formulae (1.8.18) and (1.8.19)]{G}    and following \cite{MP}  we set the following definition.
\begin{defn}
	Let $A\in \Psi_{\rm cl}(X,F)$ be a classical  (resp. a  log-polyhomogeneous  --see \cite{Le} and references therein--) operator of order $a$ with symbol $\sigma(A)\sim \sum_{j=0}^\infty\sigma_{a-j}(A)(\xi)$  (resp. $\sigma(A)(\xi)\sim \sum_{\ell=0}^{k}\sum_{j=0}^\infty\sigma_{a-j, \ell}(A)\, \log^\ell \vert \xi\vert$) with $\sigma_{a-j}(A)(\xi)$ (resp.  $\sigma_{a-j,\ell}(A)(\xi)$) homogeneous of degree $a-j$ for $\vert \xi\vert\geq 1$. We call $A$  
	{\bf geometric}, if
	in any local trivialisation, the homogeneous components
	$\sigma_{a-j}
	(A) $ (resp. $\sigma_{a-j, \ell}
	(A) $ for any $\ell\in \{0,\cdots, k\}$)
	are homogeneous of
	order
	$j$
	in the jets of the metric and of the connection.
\end{defn}
In particular, a differential operator
$A=\sum_{\vert \alpha\vert \leq a}c_\alpha(x)\, \partial_x^\alpha$ is geometric if $c_\alpha(x)$  is homogeneous of degree $j= a-\vert \alpha\vert$ in the metric and the connection on the auxillary bundle.
The Laplace-Beltrami operator (resp. the Dirac operator) associated with the metric $g$ (resp. and the connection on $W$) are geometric differential operators. More generally, the Laplace  operator (resp. Dirac   operator) associated with the connection $\nabla^F$ are  geometric differential operators. 
\begin{lem}
	An operator ${\mathcal A}\in \Psi_{\Gamma, {\rm cl}}(\wt M, \wt E)$ such that \[{\mathcal A}\sim \wt A_0\quad {\rm and}\quad A_0\sim A\]for some geometric operator   $A\in\Psi_{\rm cl}(M, E)$,  is itself geometric. In particular, with the notations of \eqref{eq:liftedOp} we have
	\[\left(A\, {\rm geometric}\, {\rm and}\, {\mathcal A}\in \wt{[ A ]_{\rm diag}}\right)\Longrightarrow \left( {\mathcal A}\, {\rm geometric} \right).\]
\end{lem}
\begin{proof}
	Since  $\sim$ preserves symbols and  symbols lift to the covering (see \eqref{eq:sigmafraktildeAzero}), we have
	\[\sigma\left({\mathcal A}\right)\sim \sigma\left(\wt A_0\right)\sim\reallywidetilde { \sigma\left(A_0\right)}\sim \reallywidetilde { \sigma\left(A\right)}.\]  The fact that the operator $A$ is geometric amounts to  the components $\sigma_{a-j,\ell}\left(A\right)$ being homogeneous of degree $j$ in the jets of the metric and the connection. Its lift $ \wt { \sigma\left(A\right)}$ obeys the  same conditions w.r.t to the metric $\wt g$ and the connection $\wt \nabla^W$ on the covering and hence so does $\sigma\left({\mathcal A}\right)$, which shows that ${\mathcal A}$ is geometric.
\end{proof}
The results of \cite{MP} relative to  holomorphic families of the type $A(z)= A\, Q^{-z}$ generalise to any holomorphic family. The proof of this more general statement can be carried out along the same lines of the proofs of  \cite[Corollary 1 and Theorem 1]{MP}.
Adopting Gilkey's notations \cite[par. 2.4]{G},  let us denote by
${\mathcal P}^{g,\nabla^W}_{n,k,p}$  
(which we write ${\mathcal P}^{g }_{n,k,p}$  in the absence of twisting)  the linear space consisting of
$p$-  form valued invariant 
polynomials
that are homogeneous of order
$k$
in the jets of the metric
and of the connection
$\nabla^W$. \begin{thm}\label{thm:geometricop} Let $A(z)\in \Psi_{\rm cl}(M,E)$ and ${\mathcal A}(z)\in \Psi_{\Gamma,{\rm cl}}(\wt M, \wt E)$ be two holomorphic families of classical pseudodifferential operators such that\[A(z)\, {\rm geometric}\, {\rm and}\, {\mathcal A}(z)\in \reallywidetilde{[ A (z)]_{\rm diag}}.\] Then
	$  {\mathcal A}(z)$ is geometric and if both  $A(0)$ and ${\mathcal A} (0)$ are differential operators,  
	\begin{enumerate}
		\item for any $x\in \wt M$,  the  residue densities  defined  in \eqref{eq:covresiduedensity}		\[{\rm Res}_{x}({\mathcal A}^\prime(0))=\reallywidetilde{{\rm Res}_{\pi(x)}(A^\prime(0))} \]
		lie in ${\mathcal P}_{n,n,n}^{\wt  g, \wt \nabla^ W}$. 
		\item Consequently, \[	{\rm fp}_{z=0}{\rm TR}_\Gamma\left(\mathcal A  (z)\right)=\frac{1}{q}\, {\rm Res}_{\Gamma}({\mathcal A}^\prime(0))=\frac{1}{ q}\, \int_{F} {\rm Res}_{ x}({\mathcal A}^\prime(0))(x)\, dx= \frac{1}{ a}\, \int_{M} {\rm Res}_{x}({\mathcal A}^\prime(0))(x)\, dx\]  is  the integral of densities generated by Pontrjagin forms on the fundamental domain and Chern forms on the auxillary bundle. 
	\end{enumerate}
\end{thm}
\begin{rk}
	  In  \eqref{eq:covresiduedensity} the  residue densities are defined for classical pseudodifferential operators whereas ${\mathcal A}^\prime(0)$  is a log-polyhomogeneous operator (see e.g. \cite{Le} and references therein) but it follows from the previous results that they extend to derivatives 
	 ${\mathcal A}^\prime(0)$ whenever ${\mathcal A} (0)$ is differential.
\end{rk}
\begin{proof}
	In order to get identities on the level of densities, we need to apply the above results to families    $\phi\, {\mathcal A}(z)$ for any $\Gamma$-invariant function $\phi$ on $\wt M$.  
	
	\begin{enumerate}\item  Since ${\mathcal A}(0)$ and $A(0)$ are differential operators, we know that  $\phi\, {\mathcal A}^\prime (0)$ and   $\phi\, A^\prime (0)$ (which are not classical operators) have well defined Wodzicki residue. The fact that the Wodzicki residue canonically lifts to coverings (Proposition  \ref{prop:ResTRliftedA}) applied to  $\phi\, {\mathcal A}^\prime (0) $  then yields for any smooth $\Gamma$-invariant function $   \phi$ on $\wt M$
		\[\int_F {\rm Res}_x\left(\phi\, {\mathcal A}^\prime(0)\right)\, dx={\rm Res}_\Gamma\left(\phi\, {\mathcal A}^\prime (0)\right)= {\rm Res} \left(\phi\, { A}^\prime (0)\right)= \int_M {\rm Res}_x\left(\phi\, A^\prime(0)\right)\, dx \]
		from 	which  we deduce the identity on the level of densities:
		\[{\rm Res}_{x}({\mathcal A}^\prime(0))=\reallywidetilde{{\rm Res}_{\pi(x)}(A^\prime(0))}. \] 
		\item 	The second statement  follows from combining   \eqref{eq:PSclassicalopdiff-bis}  which relates the finite part of  the canonical trace  to the residue applied to $\phi\, {\mathcal A}(z)$ for any $\Gamma$-invariant function $\phi$ on $\wt M$, with  Gilkey's theory of invariants  \cite[Theorem
		2.6.2]{G} since  ${\mathcal P}^{g,\nabla^W}_{n,n,n} $
		is a polynomial in the $2$-jets of the metric and the one jets of the auxillary
		connection.
	\end{enumerate}
\end{proof}
Consequently, $\zeta$-spectral  invariants for geometric operators can be written as integrals of densities generated by Pontrjagin forms on the underlying manifold and Chern forms on the auxillary bundle. 
\begin{cor} With the notations of \eqref{eq:diffliftedregtraces}, the spectral zeta invariants \[\zeta_{A_\ve ( \mathbf  D ), Q_\ve(\mathbf \Delta)}(0)=	\zeta_{{\mathcal A}_\ve(\widetilde{\mathbf  D}),{\mathfrak Q}_\e(\widetilde{\mathbf  \Delta})}(0)\]
	 can be written as integrals of densities generated by Pontrjagin forms on the underlying manifold and Chern forms on the auxillary bundle. 
\end{cor}
\begin{proof} This follows from applying Theorem \ref{thm:geometricop} to the families $ z\mapsto  A_\ve  (\mathbf D)\,Q_\ve (\mathbf \Delta)^{-z}\in \Psi_{\rm cl}(M, E) $, resp. $z\mapsto  {\mathcal A}_\ve (\wt{\mathbf D})\,Q_\ve (\wt { \mathbf \Delta})^{-z}\in \Psi_{\Gamma,{\rm cl}}(\wt M, \wt E) $. 
	\end{proof}

 %%%%%%% %%%%%%% 
%%%%%%% %%%%%%% 
%%%%  section

\newpage
%\vskip 1cm
\small{\appendix
\section{Pseudodifferential operators on open subsets of $\R^n$}
\label{app:pdoRn}
We review known results following \cite{Sc, T}, thus fixing the notation needed in the paper.

Let $U\subset \R^n$ be an open subset and $k\in \N$. 
\begin{defn} 
\begin{itemize}
\item For any $m\in \R $, let $S^m(U\times \R^n, {\rm gl}_k(\C))$  be the set of functions $\s\in C^\infty\left(U\times \R^n,{\rm gl}_k(\C)\right)$ such that 
\begin{equation}\label{eq:A1estimate}  \forall \alpha, \beta \in \Z^n_{\geq 0},  \text{ for any compact }  K\subset U, \quad \exists C_{\alpha\beta K}>0, \quad \Vert\partial^\alpha_\xi\partial^\beta_x\s(x,\xi)\Vert\leq C_{\alpha\beta K}(1+|\xi|)^{\Re(m)-|\beta|},\end{equation} where $\Vert\cdot\Vert$ is the supremum norm on ${\rm gl}_k(\C)$,
%\item \sy{We set} $S^{-\infty}(U\times \R^n,{\rm gl}_k(\C)):=\bigcap_{m\in \R}S^m(U\times \R^n,{\rm gl}_k(\C))$;
\item For any $m\in \C$, let $S^m_{cl}(U,\C^k)$ be the space of classical symbols, \emph{i.e.} the set of symbols $\s\in S^{\Re(m)} (U\times \R^n,{\rm gl}_k(\C))$ such that there exists  $\s_{m-j}=\s_{m-j}(x,\xi)$ in $ C^\infty\left(U\times \R^n\setminus \{0\}, {\rm gl}_k(\C)\right) $, $j=0,1, 2,\dots$ with $\s_{m-j}$ positively homogeneous of degree $m-j$ with respect to $\xi$ obeying the following relation
$\s(x,\xi)\sim \sum_{j=0}^\infty \s_{m-j}(x,\xi)$ \emph{i.e.}, for an excision function $\chi(\xi)$ around zero (By an excision function around a point we mean a smooth function on $\R$ which vanishes in a neighborhood of the point and is one outside a ball of radius one centered at this point) and any positive integer $N $,
\begin{equation}
\label{eq:classical}
\s(x,\xi)-\sum_{j=0}^N \chi(\xi) \,\s_{m-j}(x,\xi)\in S^{\Re(m)-N}(U\times \R^n,{\rm gl}_k(\C)).
\end{equation}
\item We call  $S^{-\infty}(U\times \R^n,{\rm gl}_k(\C)):=\cap_{m\in \R}S^{m}(U\times \R^n,{\rm gl}_k(\C))=\cap_{m\in\C} S_{\rm cl}^{m}(U\times \R^n,{\rm gl}_k(\C))$ the space of smoothing symbols on $U$.
\end{itemize}
\end{defn} 
\begin{rk}
%\begin{enumerate}
	%\item By definition, $S^m_{cl}(U\times \R^n,{\rm gl}_k(\C))\subset S^{\Re(m)}(U\times \R^n,{\rm gl}_k(\C))$  $\;\forall m\in \C$.
 For any $m\in \C$, the relation \begin{equation}\label{eq:simsymb}\sigma\sim\tau \Longleftrightarrow \sigma-\tau \in S^{-\infty}(U\times \R^n,{\rm gl}_k(\C))\Longleftrightarrow \sigma_{m-j}=\tau_{m-j}\quad \forall j\in \Z_{\geq 0}\end{equation} defines an equivalence relation on $S^m_{cl}(U\times \R^n,{\rm gl}_k(\C))$.
%\end{enumerate}
\end{rk}
By Schwartz's kernel theorem, to a  continuous linear operator $A\colon C^\infty_c(U, \C^k)\to C^\infty(U, \C^k)$  we assign its (uniquely defined) distributional  Schwartz kernel   $K_A\in {\mathcal D}^\prime(U\times E)\otimes {\rm gl}_k(\C)$. 
\begin{defn}\label{defn:PSDO} 
\begin{enumerate}
\item For any $m\in \R$ (resp. for any  $m\in \C$), a continuous linear operator $A\colon C^\infty_c(U, \C^k)\to C^\infty(U, \C^k)$ is {\bf  pseudodifferential (resp. a classical pseudodifferential operator) of order $m$} if it its kernel is an oscillatory integral of the type:
\begin{equation}\label{eq:oscint}K_A(x,y)= \frac{1}{(2\pi)^n}\int_{\R^n} e^{i\langle x-y,\xi\rangle} a(x,y,\xi)\, d\xi,\end{equation} where the amplitude $a(x,y,\xi)$ lies in the symbol space $S^m(U\times U\times \R^n, {\rm gl}_k(\C))$ (resp. $S^m_{\rm cl}(U\times U\times \R^n, {\rm gl}_k(\C))$.
\item To a given symbol $\s\in S^m(U\times \R^n, {\rm gl}_k(\C))$ one   assigns  the pseudodifferential operator 
   \begin{equation}
   \label{Op}
    (\Op(\s)u)(x)=\int_U\int_{\R^n} e^{i(x-y)\cdot \xi}\s(x,\xi)u(y)dyd\xi ,\;\; \forall u\in C_c^\infty(U,\C^k).
\end{equation} The operator $\Op(\s)$ is called {\bf classical} if $\sigma$ is.
\end{enumerate}  
\end{defn}
 \begin{rk}\label{rk:smoothkernel}\begin{enumerate}
\item Given a symbol $\sigma \in S^m(U\times \R^n, {\rm gl}_k(\C))$, the kernel $K(x,y)$ of the operator $\Op(\sigma)$ defines  a  tempered distribution in $x-y$, if moreover $\sigma \in S^{-\infty}(U\times \R^n, {\rm gl}_k(\C))$, then $K(x,y)$ is  smooth on $U\times U$.
\item Conversely \cite[Comments following Th.1]{EKS}, any smooth kernel $K$ on $U\times U$ is the kernel of a pseudodifferential operator with amplitude 
$a(x,y,\xi):=(2\pi)^n\, K(x,y)\,e^{-i\langle x-y,\xi\rangle}\, \psi(\xi)$, where $\psi$ is a smooth function with compact support in $\R^n$ such that $\int_{\R^n}\psi(\xi)\, d\xi=1$. The amplitude $a(x,y,\xi)$ lies in $S^{-\infty}(U\times \R^n, {\rm gl}_k(\C))$ since it is smooth in  $(x,y)$ and has compact support in $\xi$.
\\
\end{enumerate}
\end{rk}
The operator $\Op(\s)$ is a  properly supported (see e.g. \cite[Def. 3.6]{T}) pseudodifferential operator, a concept we now recall; 
 recall that the support of a distribution   is  the complement of the set on which the distribution vanishes. 
 \begin{defn}\label{defn:properly supported}Given an    open subset  $U$ of $\R^n$,  \begin{enumerate}  
 \item  a distribution $T\in {\mathcal D}^\prime(U\times U, {\rm gl}_k(\C)) $ is {\bf properly supported} if,  for any compact $K\subset U$, its support has compact intersection with $K\times U$ and $U\times K$  or equivalently, if  the restriction to the support of $T$  of the two canonical projection maps $U\times U\to U$ mapping $(x,y)$ to $x$ and $y$ respectively are proper maps (i.e., the preimage of a compact set is compact);
  \item 
an operator $A\colon C^\infty_c(U,\C^k)\to C^\infty( U,\C^k)$     is {\bf properly supported} if its Schwartz kernel  is.
  \end{enumerate}
\end{defn}
\begin{rk}Properly supported pseudodifferential operators stabilise $C^\infty_c(U,\C^k)$ and are therefore composable.
\end{rk}
\begin{prop}\label{prop:PsidoU} For any $m\in \R$ (resp. $m\in \C$), the  linear operator  $A\colon C^\infty_c(U, \C^k)\to C^\infty(U, \C^k)$ is a (resp. classical) pseudodifferential operator   of order $m$  if and only if it is    of the form
 \begin{equation}\label{eq:OpAT} A=\Op(\s_A)+S(A), \end{equation} where $\sigma(A)\in S^m(U\times \R^n, {\rm gl}_k(\C))$ (resp. $\sigma(A)\in S^m_{\rm cl}(U\times \R^n, {\rm gl}_k(\C))$) and $S(A)$ is an operator with a smooth kernel supported outside the diagonal of $U$.
\end{prop} 
\begin{proof}
An operator of the form $\Op(\s)$ for some $\sigma\in S^m(U\times \R^n, {\rm gl}_k(\C)), m\in \R$, is a pseudodifferential operator of order $m$ with amplitude $a(x,y,\xi)= \sigma(x,\xi)$. On the other hand, by Remark \ref{rk:smoothkernel} 2), any smooth kernel has an amplitude  which lies in $S^{-\infty}(U\times U\times  \R^n, {\rm gl}_k(\C))$. Combining these two facts, we conclude that any operator $A$ of the form (\ref{eq:OpAT}) is a psedudodifferential operator of order $m$; it  is classical if its symbol $\sigma(A)$ is classical.\\
 Conversely, let $A$ be a (resp. classical) pseudodifferential operator of order $m$. Let $\ve>0$ and  $\chi_\ve$ be a smooth function  with support containing the diagonal $\Delta_U:=\{(x,x)\in U\times U\}$ such that $\chi_\ve$ is identically one on   $\{m\in U\times U\,: d(m, \Delta_U) \leq \frac{\ve}{2}\}$ and  identically zero on  $\{m\in U\times U\,:  d(m, \Delta_U) \geq \ve\}$. We split the Schwartz  kernel $K_A$  of $A$  accordingly
 $$K_A= \chi_\ve\, K_A+(1-\chi_\ve)K_A\ .$$ 
 Both   projection maps ${\rm supp}(\chi_\ve)\to U$ are proper, so $\chi_\ve\, K_A$   is properly supported and corresponds to the kernel of some pseudodifferential operator $\Op(\s_A) $  with symbol $\sigma(A)(x,\xi)=e^{-i\langle x, \xi\rangle}\, A\, e^{i\langle x, \xi\rangle}$ in $S^m(U\times \R^n,{\rm gl}_k(\C))$ (see e.g.  \cite[Thm. 3.8]{T},  \cite[(1.5.3.8)]{Sc} and the comments that follow). The kernel $(1-\chi_\ve)K_A$ is  smooth and  supported outside the diagonal $\Delta_U$.
\end{proof} 
 
This leads to the following definition.
\begin{defn} 
\begin{itemize}
\item For any $m\in \R$, let $\Psi^m(U,\C^k)$ is the class of all linear operators $A: C_c^\infty(U,\C^k) \to C^\infty(U,\C^k)$ of the form (\ref{eq:OpAT}), with symbol $\s_A\in S^m(U,{\rm gl}_k(\C))$ and $S(A)$ an operator with a smooth Schwartz kernel $K_{S(A)}\in C^\infty(U\times U,{\rm gl}_k(\C))$ supported outside the diagonal of $U$;
\item For any $m\in \C$, $\Psi_{cl}^m(U,\C^k)$ is the class of all linear operators $A\in \Psi^m(U,\C^k)$ whose  symbol  $\s_A$ lies in $ S_{cl}^m(U\times \R^n,{\rm gl}_k(\C))$.
\end{itemize}
Let 
$$\Psi^{-\infty}(U,\C^k):= \cap_{m\in \R}\Psi^m(U,\C^k)=\cap_{m\in \C}\Psi_{cl}^m(U,\C^k)
$$
 denote the space of smoothing operators, those whose symbols lie in $S^{-\infty}(U,{\rm gl}_k(\C))$.
\end{defn}
 Let us state an easy yet very useful result, the proof of which we omit here, referring the reader to any classical textbook on pseudodifferential operators.
 \begin{cor} For any $m\in \R$ (resp. $m\in \C$), the sets
$\Psi^m(U,\C^k)$ (resp. $\Psi_{cl}^m(U,\C^k)$) are stable under summation $A\mapsto A+R$ with an operator $R$ whose kernel is smooth.  Moreover, the relation
\begin{equation}\label{eq:sim}A\sim B\Longleftrightarrow A-B\;\text{ has a smooth kernel}\end{equation} defines an equivalence relation on $\Psi^m(U,\C^k)$ (resp. $\Psi_{cl}^m(U,\C^k)$) and  for any two elements $A,B\in \Psi_{\rm cl}^m(U,\C^k)$ we have
\begin{equation}\label{eq:simop}A\sim B\Longleftrightarrow \sigma(A)\sim \sigma(B)\Longleftrightarrow \sigma_{m-j}(A)=  \sigma_{m-j}(B)\quad\forall j\in \Z_{\geq 0}.\end{equation}
\end{cor}
Whereas differential operators are local, pseudo-differential operators are not local for they smear out the supports of the sections on which they act. However, they are  pseudo-local in so far as they do not smear out  their singular supports.  The following definition captures some features of  pseudo-locality.
\begin{defn} \label{defn:Clocal}\cite[Def. 3 pag. 98]{Shu} An operator $A\colon C^\infty_c(U,\C^k)\to C^\infty( U,\C^k)$   with Schwartz kernel $K_A$  is  {\bf $C$-local} for some $C\geq 0$, if $K_A(x,y)=0$ $\forall x, y$ with $\vert x-y\vert>C$ or equivalently, if  
$$\forall  u\in C^\infty_c(U,\C^k); \quad \supp (Au) \subset \{x: d(x, \supp u) \leq C\}.$$
\end{defn}
\begin{ex}Differential operators are $0$-local   operators.
\end{ex} 
\begin{rk}
$C$-local operators are properly supported.
\end{rk}

\smallskip
\section{Coverings and $\Gamma$-Hilbert modules}
\label{sec:appHM}
 We recall   some useful definitions, referring to \cite{Sch2}. 
\begin{defn}Let $A$ be a $C^*$-algebra.\begin{enumerate}
\item An $A$-Hilbert module is a right $A$-module $V$ with an $A$-valued inner product $\langle\cdot, \cdot\rangle$ obeying the expected properties, namely it is $A$-sesquilinear, and for any $v\in V$ the expression $\langle v, v\rangle$ is a non-negative self-adjoint element in $A$ whose square root defines a norm on $V$ for which $V$ is complete. In particular,  $A^n$ is a Hilbert module for any $n\in \N$.
\item A
finitely generated projective Hilbert
$A$-module $V$ is a Hilbert $A$-module which is isomorphic as Hilbert $A$-modules to a (closed) orthogonal direct summand of $A^n$ 
for suitable $n\in \N$. In other words, there is a Hilbert
$A$-module $W$ such that $V\oplus W \simeq A^n$. 
The range of the projection  $P:A^n\to A^n$  with range $V$ and kernel $W$ is a finitely generated projective Hilbert module.
\end{enumerate} 
\end{defn}
For a given $A$-Hilbert module $V$, let ${\rm End}_{A}(V)$ be  the algebra  of Hilbert module morphisms, 
 namely  continuous  $A$-linear maps on $V$ which have an adjoint. Then   
$${\rm End}_{A}(V)\simeq {\mathcal B}(V)\otimes A $$ is a $C^*$-algebra.
If $A$ is a von Neumann algebra, so is ${\rm End}_{A}(V)$ a von Neumann algebra \cite{La}.\\

%We keep the notations of the previous paragraph.\\
We now apply the previous construction to the von Neumann algebra  of a  countable discrete group $\Gamma$, which we first define. Let $\ell_2(\Gamma)$ be the  completion of $\C\, \Gamma$ for the norm $\Vert \sum_{g\in \Gamma} a_g\, g\Vert_2=\sqrt{\sum_{g\in \Gamma} \vert a_g\vert^2}$.  {We briefly ecall that a von Neumann algebra  is a $*$-algebra of bounded operators on a Hilbert space that is closed in the weak operator topology and contains the identity operator. The weak operator topology is  the weakest topology on the set of bounded operators on a Hilbert space, such that the functional sending an operator $T $ to the complex number $<Tx, y>$ is continuous for any vectors $x$ and $y$ in the Hilbert space.} The group  $\Gamma$ acts on $\ell_2(\Gamma)$  by the right regular representation $h\cdot\left( \sum_{g\in \Gamma} a_g\, g\right)= \sum_{g\in \Gamma} a_g\, gh^{-1}$. Equipped with the inner product $$\left\langle \sum_{g\in \Gamma} a_g\, g,\sum_{g\in \Gamma} b_g\, g\right\rangle:=\sum_{g\in \Gamma} a_g\, \overline {b_g},$$
 $\ell_2(\Gamma)$ is a Hilbert space and via the right regular representation, $\C \G$ can be viewed as a subalgebra of the $C^*$-algebra ${\mathcal B}\left(\ell_2(\Gamma)\right)$ of bounded linear operators on $\ell_2(\Gamma)$. The group von Neumann algebra
${\mathcal N}\Gamma$  is the closure of
$\C\,\Gamma$  in
${\mathcal B}\left(\ell_2(\Gamma)\right)$ with respect to the weak operator topology.
\begin{ex} Take  $\Gamma=\mathbb Z^n$. The Fourier transform gives an isometric $\mathbb Z^n$-equivariant isomorphism $\ell^2(\mathbb Z^n)\to L^2(\T^n)$, where $\T^n$ is the $n$-dimensional torus. 
Therefore $\mathcal N\mathbb Z^n$ coincides with the commutant $\mathcal L(L^2(T^n))^{\mathbb Z_n}$ of the $\mathbb Z_n$-action on $L^2(\T^n)$, and one obtains an isomorphism  $\mathcal N\mathbb Z^n\simeq L^\infty(\T^n)$. 
\end{ex}  
\begin{defn} 
 \begin{enumerate}
 \item The {\bf Kaplansky trace}, also called {\bf $\Gamma$-trace} on ${\mathcal N }\Gamma$ is  the linear form ${\rm tr}_\Gamma: {\mathcal N} \Gamma\to \C$ defined as $${\rm tr}_\Gamma(a)=\langle a\,  e, e\rangle= a_e,$$
 where $a =  \sum_{g\in \Gamma} a_g\, g$ and $e\in \C\Gamma$ is the element with all components zero outside the first one which is one. 
 \item The  $\Gamma$-{\bf dimension} of a finitely generated $\Gamma$-module $V$ represented by an idempotent matrix $P=(p_{ij})\in {\rm gl}_n\left({\mathcal N}\Gamma\right)$ is   $$ {\rm dim}_\Gamma(V):=\sum_{i=1}^n {\rm tr}_\Gamma (    p_{ii})=  {\rm tr}_\Gamma \left(  {\rm tr}_{{\rm gl}_n(\C)}(  P)\right), $$ where ${\rm tr}_{{\rm gl}_n(\C)}$ is the matrix trace given by the sum of the diagonal elements of the matrix with complex entries.
\item  More generally, the $\Gamma$-trace extends to  any operator $T\in {\rm End}_{{\mathcal N} \Gamma}(V)\simeq {\mathcal B}(V)\otimes {\mathcal N} \Gamma$ by
$${\rm tr}_\Gamma(T) := {\rm tr}_\Gamma\left( {\rm tr}_{{\rm gl}_n(\C)}(TP)\right).$$
\end{enumerate}
\end{defn} 
\begin{ex}
 The Kaplansky   trace on  $\mathcal N\mathbb Z^n\simeq L^\infty(\T^n)$ is  given by  the integration map against the canonical volume measure on the flat torus:
\begin{equation}\label{ex:tauZn}
\tau(f)=\int_{\T^n}fd\mu. 
\end{equation}
\end{ex}
Let $M$ be a closed manifold with $\pi_1(M)=\Gamma$, let $\pi\colon \widetilde M\to M$ be a universal covering of $M$, with $\Gamma$ acting on the right by deck transformations. Then $ \ell^2(\Gamma)$ is a finite type $\mathcal N\Gamma$-module for $\ell^2(\Gamma)\simeq (\mathcal N\Gamma)_2$, the $L^2$-closure of $\mathcal N\Gamma$. More generally, we have
\begin{lem}Let $\wt M$ be a $\Gamma$-covering of $M$. The bundle $\mathcal H\to M$   defined by
$$
\mathcal H=\wt M\times_\G \ell^2\Gamma  
$$
  is a flat bundle of finitely generated projective $\mathcal N \Gamma$-modules. 
  \end{lem}
  \begin{proof}
This follows from the fact that the left $\Gamma$-action   and the right $\mathcal N\Gamma$-action   on $\ell^2(\Gamma)$ commute so $\mathcal H:=\widetilde M\times_\Gamma \ell^2 (\Gamma)$ is a finitely generated projective bundle of (right) $\mathcal N\Gamma$-Hilbert modules over $M$. Moreover, $\mathcal H$ is endowed with a flat structure since the transition functions are locally constant.
  \end{proof}

There is a well known dictionary between the space of $L^2$-sections of $\widetilde E$ on $\widetilde M$ and the sections of the bundle $E$ twisted by $\mathcal H$, for which we refer for example to \cite[7.5]{Sch2} or \cite[Prop. E.6]{PzSc} for full details. 

Let us first recall that the space $C_c(\widetilde M, \widetilde E)$ is a right $\mathbb C\Gamma$-module with structure given by $(\xi\cdot s)(\tilde x)=\sum_{g\in \Gamma} (R^*_g \xi )(\tilde x)s(g^{-1})$, where $R^*_g$ denotes the pullback map, $s$ is in $ \C\Gamma$ and $\xi\in C_c(\widetilde M, \widetilde E)$. By completion, this endows  $L^2(\widetilde M,\widetilde E)$ with the structure of right $\mathcal N\Gamma$-module. 
\begin{defn}
Let $s$  be a section of $\widetilde E$. Define the section  
$\Phi( s)$  of $E\otimes \mathcal H$ by
$$\Phi(s)(x)=\hat s(x)=\sum_{\gamma \in \Gamma} s(\gamma \tilde x)\otimes [\tilde x, \gamma]$$
 where $\tilde x$ is an arbitrary lift of $x\in M$, and $\widetilde E_{ \gamma \tilde x} $ is identified with  $E_{x}$.
 \end{defn}
The map $s\mapsto \hat s=\Phi(s)$ identifies $\{s\in C_c^\infty(\widetilde M, \widetilde E ) \,|\, \sum_\gamma |s(\gamma x)|^2<\infty\}$ with $ 	C^\infty\left( M, E\otimes \mathcal H\right)$.
It extends to an isometry $\ds\Phi: L^2(\tilde M,\widetilde E)\to L^2( M, E\otimes \mathcal H) $  of $\mathcal N\Gamma$-Hilbert modules.
\begin{ex}
\label{ex:covTn}
To the universal covering $\pi\colon \R^n \to \T^n$ with fundamental group $\Z^n$ corresponds the finitely generated projective bundle $\mathcal H:=\R^n\times_{\Z^n} \ell^2 (\Z^n)$   of (right) $L^\infty(\T^n)$-Hilbert modules over $\T^n$.  In this correspondence, $L^2$-functions on $\R^n$ are viewed as  $L^2$-sections of the bundle ${\mathcal H}$ over $\T^n$ via the map which sends $f$ to $\hat f\colon x \mapsto \sum_{\gamma \in \Z^n} f(\gamma \pi^{-1}(x))\otimes [\pi^{-1}(x), \gamma]$.  This induces  an isometry 
\begin{equation}\label{PHI}
\Phi\colon  L^2(\R^n)\simeq L^2(\T^n)\otimes \ell^2(\Z^n)\longrightarrow L^2(\T^n,   \mathcal H)
\end{equation}
which sends $\{f\in \Ci(\R^n) : \sum_{\gamma \in \Z^n} |f(\gamma x)|^2 <\infty \; \forall x\in \R^n \}$ to $\Ci(\T^n,\mathcal H)$.
 \end{ex} 

\begin{defn} To any first-order differential operator $D$ on $C^\infty(M, E)$, one can assign the first order differential operator $D_{\mathcal H}$
defined on  a section $\hat s(x):= \sum_\gamma \tilde s(\gamma x)\otimes  [\tilde x,\gamma]$ of $E\otimes \mathcal H$  as \begin{equation}\label{eq:DH}
D_{\mathcal H}\left(\hat s\right)(x):=\sum_{\gamma\in \Gamma}\tilde D\tilde s(\gamma \tilde x)\otimes [\tilde x,\gamma],
\end{equation}
locally on an evenly  covered neighborhood $U$ of $x\in M$, with lift $\widetilde U$ an open set such that $\pi_{U}$ is a diffeomorphism, so that  $U\ni y\mapsto [\tilde y, \gamma]$ is a flat section of $\mathcal H$ for all $\gamma \in \Gamma$.
 \end{defn} 
The following proposition is proved in \cite{Sch2} (see also \cite[Proposition 3.12]{BP}). We sketch the proof in this appendix for the sake of completeness. 
\begin{prop}\label{prop:dict}(\cite{Sch2}) The map  $$\Phi: L^2(\tilde M,\widetilde E)\to L^2( M, E\otimes \mathcal H) $$ yields  an isometry of $\mathcal N\Gamma$-Hilbert modules 
 which   for any first-order differential operator $D$ on $C^\infty(M, E)$ satifies
 \begin{enumerate}
\item \label{eq:relationH}
$\Phi^\sharp  \wt D:=\Phi \wt D\Phi^{-1}= D_{\mathcal H}$,
\item \label{eq:funct-calc} $\Phi^\sharp \left(h(\wt D)\right)=h(D_\mathcal H)$, for every bounded measurable function $h:\R\to\R$. 
\end{enumerate}
 \end{prop}
\begin{proof}
Relation \eqref{eq:relationH} follows from \eqref{eq:DH} and hence \eqref{eq:funct-calc} follows observing that since the unitary equivalence of the self-adjoint unbounded operators $\wt D$ and $D_{\mathcal H}$ implies the same property for all bounded measurable functions of the latter,  using functional calculus.
\end{proof}
\smallskip 
\section{The groupoid interpretation}
\label{sec:groupoids}
We use the notations of the previous appendix. An equivalent description of the calculus of $\Gamma$-invariant operators on the covering $\wt M$ involves the fundamental groupoid $G$ defined as the quotient space $$ G:=(\wt M\times \wt M)/\Gamma$$ of $\wt M\times \wt M$ by the diagonal action
%\begin{equation}\label{eq:diagonal}
$
\gamma\cdot (\tilde x,\tilde y ):=(\gamma\cdot\tilde x,\gamma\cdot\tilde y )
$.
%\end{equation} 
The groupoid structure on $G$ is given  in terms of the space of units is  $G^{(0)}=M$  and the source and range maps  
$$
s [(\tilde x,\tilde y )]:=\pi(\tilde x)\;,\;\;\;\;\; r[(\tilde x,\tilde y )]:=\pi(\tilde y)\ .
$$
For further details on the structure of this groupoid we refer to \cite[Section 2]{BP} in  the more general setting of  the monodromy groupoid of a foliated bundle.

\medskip

Pseudodifferential calculus on groupoids was introduced by Connes \cite{Co}, Nistor--Weinstein--Xu \cite{NWX},  Monthubert--Pierrot \cite{MP} and by Vassout \cite{Vas} in relation with the Wodzicki residue. 
We will compare it here with the calculus developed by Vassout which bears the advantage of including a large subalgebra of non properly supported smoothing operators and is therefore well suited for the construction of complex powers \cite{Vas}.

Let  $G$ be a differentiable groupoid and $E\to G^{(0)}$ a vector bundle on its space of units, and for any $x\in G^{(0)}$ let $G_x:=s^{-1}(x)$.

One considers  families $A=(A_{x})_{x\in M}$, where $A_x$ is an operator acting on $C_c^\infty(G_x, r^*(E))$, satisfying the $G$-invariance condition
\begin{equation}\label{eq:GinvG}
A_{r(g)}U_g=U_g A_{d(g)}
\end{equation}
where $U_g: C^\infty(G_{s(g)}, r^*(E))\to C^\infty(G_{r(g)}, r^*(E))$ is $(U_g f)(g'):=f(g'g)$.
The support of a  family $A$ is is a subset of $\{(g,h)\in G\times G , s(g)=s(h)\}$ defined as
$$
\supp (A):=\bigcup_{x\in G^{(0)}}\supp (A_x).
$$  
\begin{defn}{}\cite[Definition 8]{NWX}
The family $A$ is said to be 
\begin{itemize}
\item \emph{properly supported}, if $p_i^{-1}(K)\cap \supp  (A)$ is compact for any compact subset $K\subset G$, where $p_1, p_2:G\times G\to G$ are the projections on the two factors;
\item  \emph{compactly supported}, if $\supp (A)$ is compact;
\item \emph{uniformly supported}, if $\mu_1(\supp (A))$ is compact (where $\mu_1(g, h):=gh^{-1}$).
\end{itemize}
\end{defn}

The algebra of $G$-pseudodifferential operators
 (resp. classical $G$-pseudodifferential operators in the sense of Vassout) consists of two building blocks: 
 %\sy{The regularity of the families of pseudodifferential operators is missing, $C^k$ in \cite{Vas}}: 
\begin{enumerate}
	\item[a)] an algebra $\Psi_{ \rm cpt}( G, E)$  whose elements are smooth (The degree of regularity can be chosen according to the needs) families $A=(A_{x})_{x\in M}$ where $A_x$ is a   (resp. classical) pseudodifferential operator acting on $C_c^\infty(G_x, r^*(E))$, such that $A$ is $G$-invariant and compactly  supported
	\item[b)] an algebra $\Psi^{-\infty}( G, E)$ of \emph{smoothing} operators  defined by means of a scale of Sobolev modules $H^s( \mathcal W)$, $s\in \R$, on the groupoid $C^*$-algebra  \cite[Proposition 4.2.5]{Vas} 	as   the intersection 
\begin{equation}
\label{eq:C*sobolev}
\Psi^{-\infty}( G)=\cap_{s,t} {\mathcal L}\left(H^s( \mathcal W), H^t( \mathcal W)\right)
\end{equation}
of the spaces of linear maps which take a Sobolev module $H^s( \mathcal W)$ to another Sobolev module  $H^t(\mathcal W)$. 
\end{enumerate}
\begin{defn}A (resp. classical) pseudodifferential operator in the sense of \cite{Vas} consists of a sum 
$$A=A_0+S(A), \quad A_0\in \Psi_{\rm cpt}( G, E), \;S(A)\in \Psi^{-\infty}( G, E),$$
(resp. with $A$ classical).
The linear space generated by   (resp. classical) pseudodifferential operators of order  $m \in \R$ (resp. $m\in \C$) is denoted by $\Psi^m(G, E)$ (resp. $\Psi^m_{\rm cl}(G, E)$) and the whole algebra generated by such operators by $\Psi(G, E)$ (resp. $\Psi_{\rm cl}(G, E)$).
\end{defn}
\begin{ex} In the case  of the pair groupoid $G=M\times M$  of a closed smooth manifold $G^{(0)}:=M$, $\Psi_{ \rm cpt}( G, E)$ corresponds to the set of  compactly supported  operators on $M$, which  in particular are properly supported \cite[Example 1]{NWX}, so $\Psi^m(G, E)=\Psi^m(M, E)$ for any $m\in \R$.
\end{ex}
We  now focus on the case  $G=(\wt M\times \wt M)/\Gamma$, and summarize  the relations between pseudodifferential operators on $\wt M$ of Section \ref{subsec:ups} with   pseudodifferential operators on  $ G$ given in \cite[Example  4]{NWX}, \cite[\textsection 6]{Vas}, see also \cite[\textsection 3.2]{BP}.

Each fibre $  G_x:=s^{-1}(x)$ can be identified to $\wt M$ via the map
\begin{equation}\label{eq:G_x}
  \rho_x:\wt M\to G_x \;\;\;,\, \;\tilde p\mapsto [\tilde x_0,\tilde p],
\end{equation}
where $\tilde x_0$ is a fixed element in $\pi^{-1}(x)$. 
The identification of $G_x$ with $\wt M$  is unique up to the action of $\Gamma$.

Given a family $A=(A_x)_{x\in M}$, the $G$-invariance condition \eqref{eq:GinvG} implies that $A_x$ is $\Gamma$-invariant. 
Hence the family $A$ can be identified with a $\Gamma$-invariant operator on $\wt M$.

 \begin{prop} {}\cite[Example  4]{NWX}, \cite[6]{Vas}
\label{prop:Gvscov}
Let $G=(\wt M\times \wt M)/\Gamma$. One has the following identifications:
\begin{enumerate}
\item   For any $m\in \R$ (resp. $m\in\C$)
$$ \Psi_{\rm cpt}^m( G, E)\simeq \{A\in \Psi_{\Gamma}^m (\wt M, \wt E), \,A \,\text{ is propertly supported}\}, $$
$$\quad{\rm resp.}\quad \Psi_{\rm cpt}^m( G)\cap \Psi_{\rm cl}(G, E)\simeq \{ A\in \Psi_{\Gamma}^m (\wt M, \wt E)\cap  \Psi_{\rm cl}(\wt M,\wt E)\, , \,A \,\text{ is propertly supported}\},$$ 
\item  $\Psi^{-\infty}( G, E)\simeq {\mathcal S}\Psi_{\Gamma}^{-\infty}(\wt M, \wt E)  $.
\item Consequently, $\Psi_{\rm cl}^m( G, E)\simeq \Upsi_{{\rm cl},\Gamma}^m (\wt M, \wt E)$.
\end{enumerate}
\end{prop}}

\vskip .4cm
\noindent

%\subsection*{Acknowledgements}
%\newpage
 %%%%%%%%%

\end{document}